\documentclass[a4paper,fleqn,longmktitle]{cas-sc}

\usepackage[numbers]{natbib}

\usepackage{graphicx}%
\usepackage{multirow}%
\usepackage{amsmath,amssymb,amsfonts}%
\usepackage{amsthm}%
\usepackage{mathrsfs}%
\usepackage[title]{appendix}%
\usepackage{xcolor}%
\usepackage{textcomp}%
\usepackage{manyfoot}%
\usepackage{cleveref}
\usepackage{booktabs}%
\usepackage{algorithm}%
\usepackage{algorithmicx}%
\usepackage{algpseudocode}%
\usepackage{listings}%
\usepackage{enumitem}
\usepackage{float}

\theoremstyle{plain}
\newtheorem{theorem}{Theorem}[section]
\newtheorem{lemma}[theorem]{Lemma}
\newtheorem{proposition}[theorem]{Proposition}
\newtheorem{corollary}[theorem]{Corollary}
\newtheorem{assumption}[theorem]{Assumption}

\theoremstyle{definition}
\newtheorem{definition}[theorem]{Definition}

\theoremstyle{remark}
\newtheorem{remark}[theorem]{Remark}

\numberwithin{equation}{section}

\usepackage{enumitem}

\def\tsc#1{\csdef{#1}{\textsc{\lowercase{#1}}\xspace}}
\tsc{WGM}
\tsc{QE}

\ExplSyntaxOn
\cs_if_exist:NF \vbox_unpack_clear:N
  {
    \cs_new_protected:Npn \vbox_unpack_clear:N #1
      {
        \vbox_unpack:N #1
        \box_clear:N #1
      }
  }
\ExplSyntaxOff

 

\ExplSyntaxOn
\cs_gset:Npn \__first_footerline:
  {
    \group_begin:
    \small \sffamily
    \hfill \thepage \hfill 
    \group_end:
  }
\ExplSyntaxOff

\usepackage{fancyhdr}
\begin{document}

\shorttitle{Optimal Harvesting with Environmental Feedback}

\shortauthors{L. S. Wang and J. Yu}  

\title [mode = title]{Optimal Harvesting of Size-Structured Populations with Environmental Feedback and Fixed Recruitment Flux}

\author[1]{Louis Shuo Wang}
\cormark[1]  
\fnmark[1]  
\ead{wang.s41@northeastern.edu}

\affiliation[1]{organization={Department of Mathematics, Northeastern University},
            city={Boston},
            postcode={02115},
            state={MA},
            country={United States}}
            
\author[2]{Jiguang Yu}
\fnmark[1] 
\ead{jyu678@bu.edu}

\affiliation[2]{organization={College of Engineering, Boston University},
            city={Boston},
            postcode={02215},
            state={MA},
            country={United States}}
            
\cortext[1]{Corresponding author}

\fntext[1]{These authors contributed equally to this work as co-first authors.}

\begin{abstract}
We study a nonlinear size-structured transport model with distributed harvesting and prescribed recruitment flux, where environmental feedback is determined by a scalar population functional. After establishing global well-posedness on $L^1$, we reduce stationary equilibria to a scalar closure equation. This reduction reveals that loss of equilibrium uniqueness occurs through a generic fold, mathematically characterizing critical population transitions. On uniformly nonresonant equilibrium branches, we prove the existence of optimal stationary harvesting policies via the direct method of the calculus of variations. We then derive a boundary-corrected adjoint equation and establish an identity equating equilibrium sensitivity with the adjoint loop gain. This relation yields explicit criteria for the persistence and creation of optimal harvesting thresholds. Collectively, these results provide a unified analytical framework connecting environmental feedback, equilibrium structure, and optimal harvesting.
\end{abstract}


\maketitle

\section{Introduction}
\label{sec:introduction}

Renewable-resource management often depends not only on total population
abundance, but also on the distribution of individuals across biologically and
economically relevant traits. In forestry, stem diameter determines both
market value and future stand structure; in fisheries, body length affects
reproductive output, legal harvestability, and yield; and in cell populations,
size or biomass may correlate with division, death, and treatment response.
These settings naturally lead to size-structured harvesting models in which
the control is size selective, while the population itself modifies growth and
mortality through crowding, resource depletion, or aggregate biomass.

Let
\[
    \Omega=[l_0,l_m]\Subset(0,\infty),
    \qquad
    0<l_0<l_m<\infty,
    \qquad
    x(t,\cdot)\in L^1_+(\Omega),
\]
and let the scalar output
\[
    E(t)=\mathcal E(x(t))
    :=
    \langle\chi,x(t)\rangle
    =
    \int_\Omega \chi(l)x(t,l)\,dl
\]
feed back into the vital rates. We study the closed-loop size-structured
transport system
\begin{equation}
    \label{eq:intro_state}
    \partial_t x(t,l)
    +
    \partial_l\!\big(g(E(t),l)x(t,l)\big)
    =
    -\big(\mu(E(t),l)+u(t,l)\big)x(t,l),
    \qquad
    (t,l)\in(0,\infty)\times\Omega,
\end{equation}
with prescribed inflow flux
\begin{equation}
    \label{eq:intro_flux_bc}
    g(E(t),l_0)x(t,l_0)=p(t),
    \qquad
    t>0.
\end{equation}
Here
\[
    g(E,l)>0,
    \qquad
    \mu(E,l)\ge0,
    \qquad
    0\le u(t,l)\le u_{\max},
    \qquad
    p(t)>0,
    \qquad
    \chi\ge0,\quad \chi\not\equiv0.
\]
Thus the state enters the dynamics only through the scalar observation
\[
    x
    \longmapsto
    E=\langle\chi,x\rangle
    \longmapsto
    \big(g(E,\cdot),\mu(E,\cdot)\big),
\]
and the controlled equation may be viewed as a positive infinite-dimensional
control system with scalar nonlinear feedback.

Physiologically structured population models have become a standard framework
for describing populations whose dynamics depend on individual traits such as
age or size. Beginning with the pioneering works of Metz and Diekmann
\cite{metz2014dynamics} and the classical nonlinear age-dependent theory of
Gurtin and MacCamy \cite{gurtin1974non}, the mathematical theory has developed
into a mature subject through the monographs
\cite{webb1985theory,iannelli1995mathematical,wang2025analysis,perthame2007transport,magal2018theory,cushing1998introduction}.
These models have found broad applications in cellular biology
\cite{bonnet2020multiscale,yu2026microscopic,doumic2011structured,liang2025global,grabosch2020production,wang2026damage,wang2025analysis1},
ecological modelling \cite{thieme1988well,tian2018global,liu2025bidirectional},
and industrial processes
\cite{coron2012controllability,li2022controllability,shang2011analysis,wang2026algebraic}.

Among these models, nonlinear size-structured transport equations with
environment-dependent growth have received considerable attention
\cite{calsina1995model,kato2004general,yu2026rigorous,jie2026optimal,fotso2023optimal}.
Existing studies have established well-posedness
\cite{kato2004general,ackleh2016general,ainseba2022population,yu2026pattern,ainseba2020adaptative,clement2024well},
measure-valued formulations
\cite{ackleh2005measure,franco2022integral,franco2023modelling,wang2026breakdown},
and stability properties
\cite{farkas2007stability}.
In particular, the reduction of stationary equilibria to finite-dimensional
closure equations has been systematically investigated by
Diekmann, Gyllenberg, and Metz \cite{diekmann2003steady}.
Our model differs from these formulations in that the inflow boundary condition
prescribes a recruitment flux rather than a boundary density, which generates a
nontrivial boundary contribution in the adjoint system.
The associated trace and inflow boundary issues have recently been studied by
Scott and Pollock \cite{scott2022transport,yu2026beyond}.
Since our analysis is carried out in a bounded-variation framework,
the \(BV\) techniques employed throughout the paper follow the standard theory
developed in
\cite{ambrosio2000functions}.

Optimal control of structured population models has also been extensively
studied.
Pontryagin-type maximum principles were initiated by
Brokate \cite{brokate1985pontryagin,wang2026elliptic},
while infinite-dimensional optimal control theory is systematically developed
in
\cite{li2012optimal,anita2013analysis,wang2025multi}.
For structured harvesting problems, optimality conditions and bang--bang
policies have been investigated for both age- and size-structured systems
\cite{feichtinger2003optimality,park1998optimal,barbu1999optimal,
kato2008maximum,liu2025optimal,gao2022rolling}.
In particular, Kato established Pontryagin-type maximum principles for
size-structured harvesting problems \cite{kato2008maximum}, while
Hritonenko \emph{et al.}
derived existence, uniqueness, and bang--bang harvesting regimes for exploited
stationary populations
\cite{hritonenko2008maximum,hritonenko2009bang}.
For infinite-horizon formulations, rigorous treatments of adjoint equations and
transversality conditions were developed by
Aseev and Veliov \cite{aseev2012maximum},
with subsequent extensions to age-structured population models by
Skritek and Veliov \cite{skritek2015infinite}.

A central feature of \eqref{eq:intro_state}--\eqref{eq:intro_flux_bc} is that
the lower boundary condition prescribes a physical flux, not a density. Namely,
$\displaystyle x(t,l_0)=\frac{p(t)}{g(E(t),l_0)}$.
Consequently,
\[
    \delta_E\big(g(E,l_0)x(E,l_0)\big)=0,
    \qquad
    \delta_E x(E,l_0)
    =
    -\frac{p\,\partial_Eg(E,l_0)}{g(E,l_0)^2}.
\]
This boundary sensitivity is invisible if one treats \(x(t,l_0)\) as prescribed
data. It is precisely this fixed-flux structure that produces the boundary
correction in the adjoint equation. 
Transport equations with nontrivial inflow
data require careful trace and interface handling.

For stationary controls,
the equilibrium problem reduces to the scalar closure equation
\[
E=\Phi_u(E).
\]
The corresponding closure derivative determines whether stationary
equilibria persist under perturbations and whether fold bifurcations may
occur.

We next consider the stationary harvesting problem over a selected stationary
sheet
\[
u\longmapsto E_u^*,
\qquad
E_u^*=\Phi_u(E_u^*),
\]
where the harvesting objective is
\[
\mathcal Y(u)
=
\int_\Omega \pi(l)u(l)x_u^*(l)\,dl,
\qquad
x_u^*:=x_{E_u^*,u},
\]
with admissible controls belonging to a compact \(BV\)-class.
Under a uniform nonresonance condition, the stationary sheet depends
continuously on the control, which enables the application of the direct
method of the calculus of variations and yields the existence of an optimal
stationary harvesting policy.

The first-order optimality system involves a boundary-corrected rank-one
adjoint equation induced by the fixed-flux boundary condition. Its solution
admits the decomposition
\[
S
=
S_{\rm red}
-
\Gamma_0\psi_0,
\qquad
\Gamma_0
=
\frac{A_0}{1-B_0},
\]
where \(S\) denotes the switching function and
\(S_{\rm red}\) is the reduced switching function obtained from the uncoupled
adjoint problem. The scalar coefficient \(\Gamma_0\) measures the correction
generated by the environmental feedback through the rank-one coupling.

A central result of the paper establishes that the same scalar quantity
governing the closure equation also determines the adjoint correction.
Consequently, the scalar closure relation provides a unified description of
stationary equilibrium continuation, adjoint solvability, and the qualitative
structure of bang--bang harvesting policies. In particular, the above
decomposition yields explicit criteria for threshold persistence and for the
creation of new harvesting windows as the closure margin approaches zero.

The main contributions of this paper are fourfold. First, we establish global
well-posedness and prove that the controlled transport equation defines a
forward-complete positive control system on \(L^1_+(\Omega)\). Second, we
reduce stationary equilibria to a scalar closure equation and show that its
derivative characterizes equilibrium persistence and fold resonance. Third,
under a uniform nonresonance condition, we prove the existence of an optimal
stationary harvesting policy over a compact \(BV\)-admissible class. Finally,
we derive a boundary-corrected rank-one adjoint equation and establish the
fundamental connection between closure resonance, adjoint resonance, and
harvesting threshold transitions.

The remainder of the paper is organized as follows. Section~2 introduces the
mathematical model together with the assumptions and stationary formulation.
Section~3 presents the main theoretical results. Section~4 contains the proofs
and technical arguments. Finally, Section~5 concludes the paper with a
discussion of the results and possible future directions.

\section{Methods (Model)}
\label{sec:model}
\subsection{Controlled size-structured dynamics}
\label{subsec:model_dynamics}
Let \(\Omega:=[l_0,l_m]\Subset(0,\infty)\), \(0<l_0<l_m<\infty\),
\(X:=L^1_+(\Omega)\). For \(x(t,\cdot)\in X\), \(\chi\in L^\infty_+(\Omega)\),
\(\chi\not\equiv0\), define the scalar environmental output
\[
    E(t):=\mathcal E(x(t))
    :=
    \langle\chi,x(t)\rangle
    =
    \int_{l_0}^{l_m}\chi(l)x(t,l)\,dl,
    \qquad
    0\le E(t)\le \|\chi\|_{L^\infty}\|x(t)\|_{L^1}.
\]
The controlled dynamics are
\begin{equation}
    \label{eq:model_state}
    \partial_t x(t,l)
    +
    \partial_l\!\big(g(E(t),l)x(t,l)\big)
    =
    -\big(\mu(E(t),l)+u(t,l)\big)x(t,l),
    \quad
    (t,l)\in(0,\infty)\times\Omega,
\end{equation}
with inflow flux and initial condition
\begin{equation}
    \label{eq:model_flux_bc}
    g(E(t),l_0)x(t,l_0)=p(t),\quad t>0,
    \qquad
    x(0,l)=x_0(l),\quad l\in\Omega.
\end{equation}
The boundary condition is flux-type: \(x(t,l_0)=p(t)/g(E(t),l_0)\),
\(\delta_E(g(E,l_0)x(E,l_0))=0\), and formally
\(\delta_E x(E,l_0)=-p(t)\partial_Eg(E,l_0)/g(E,l_0)^2\).
The nonconservative form is
\[
    \partial_t x
    +
    g(E(t),l)\partial_lx
    =
    -\big(\partial_lg(E(t),l)+\mu(E(t),l)+u(t,l)\big)x,
\]
so that along a time characteristic
\(\frac{d}{ds}L(s;t,l)=g(E(s),L(s;t,l))\), \(L(t;t,l)=l\), the density obeys
\(\frac{d}{ds}x(s,L)=-(\partial_lg+\mu+u)(E(s),L)x(s,L)\). Thus the
time-characteristic survival rate is \(\partial_lg+\mu+u\), whereas the
stationary size-hazard is \((\mu+u)/g\).

As clearly visualized in \Cref{fig:boundary_flux}, the delicate interface created by the prescribed flux boundary condition $g(E(t), l_0)x(t, l_0) = p(t)$ geometrically separates the characteristic branches. This separation fundamentally supports the frozen-flow stability estimates and clarifies the $BV$ composition estimates.

\subsection{Assumptions and input classes}
\label{subsec:model_assumptions}
\begin{assumption}[Coefficients]
\label{ass:model_coefficients}
There exists \(g_m>0\) with \(g(E,l)\ge g_m\) on \(\mathbb R_+\times\Omega\).
For each \(E_M>0\) there is \(C_M>0\) with
\[
    |g|+|\mu|+|\partial_Eg|+|\partial_lg|+|\partial_{ll}g|
    +|\partial_E\mu|+|\partial_l\mu|
    \le C_M
    \quad
    \text{a.e. on }[0,E_M]\times\Omega,
\]
so in particular \(g(E,\cdot)\in W^{2,\infty}(\Omega)\) uniformly on bounded
\(E\)-intervals. For each \(E_M>0\) there is \(L_M>0\) with
\begin{equation}
    \label{eq:model_E_lip}
    \|g(E_1,\cdot)-g(E_2,\cdot)\|_{W^{2,\infty}(\Omega)}
    +
    \|\mu(E_1,\cdot)-\mu(E_2,\cdot)\|_{W^{1,\infty}(\Omega)}
    \le
    L_M|E_1-E_2|
\end{equation}
for all \(E_1,E_2\in[0,E_M]\), and \(E\mapsto\partial_Eg(E,\cdot)\),
\(E\mapsto\partial_E\mu(E,\cdot)\) are continuous from \([0,E_M]\) into
\(L^\infty(\Omega)\).
\end{assumption}
\begin{assumption}[Environmental weight]
\label{ass:model_chi}
\(\chi\in L^\infty_+(\Omega)\), \(\chi\not\equiv0\). For stationary
differentiability and adjoint identities, \(\chi\in W^{1,\infty}(\Omega)\).
\end{assumption}
For dynamic inputs, fix \(u_{\max}>0\), \(M_u^{\rm dyn}>0\),
\(0<p_{\min}\le p_{\max}<\infty\), and set
\[
    \mathcal U_{\rm dyn}
    :=
    \Big\{
    u\in L^\infty_{\rm loc}(\mathbb R_+;BV(\Omega)):
    0\le u\le u_{\max}\ \text{a.e.},\
    \operatorname*{ess\,sup}_{t\ge0}\operatorname{TV}(u(t,\cdot))
    \le M_u^{\rm dyn}
    \Big\},
\]
\[
    \mathcal P
    :=
    \Big\{
    p\in L^\infty_{\rm loc}(\mathbb R_+;[p_{\min},p_{\max}])
    \cap BV_{\rm loc}(\mathbb R_+):
    \operatorname{TV}_{[0,T]}(p)\le M_p(T)\ \forall T>0
    \Big\},
\]
and \(\mathcal W:=\mathcal U_{\rm dyn}\times\mathcal P\),
\(w=(u,p)\in\mathcal W\). The initial state satisfies
\(x_0\in BV_+(\Omega):=\{x\in BV(\Omega):x\ge0\ \text{a.e.}\}\). If traces are
used classically, impose
\begin{equation}
    \label{eq:model_corner_compatibility}
    g(E_0,l_0)x_0(l_0)=p(0),
    \qquad
    E_0:=\int_\Omega\chi(l)x_0(l)\,dl;
\end{equation}
for the weak theory \eqref{eq:model_corner_compatibility} is not imposed in the
state space. The state metric is
\(d_X(x,\widetilde x):=\|x-\widetilde x\|_{L^1(\Omega)}\), and
\(|\mathcal E(x)-\mathcal E(\widetilde x)|\le\|\chi\|_{L^\infty}
\|x-\widetilde x\|_{L^1}\).
\subsection{Weak solutions and control-system notation}
\label{subsec:weak_solutions}
Let \(T>0\), \(x_0\in X\), \(w=(u,p)\in\mathcal W\). A function
\(x\in C^0([0,T];L^1_+(\Omega))\) is a weak solution of
\eqref{eq:model_state}--\eqref{eq:model_flux_bc} on \([0,T]\) if
\(E(t)=\langle\chi,x(t)\rangle\), \(x(0)=x_0\) in \(L^1\), and for all
\(\varphi\in C^1([0,T]\times\Omega)\) with \(\varphi(t,l_m)=0\) and all
\(t\in[0,T]\),
\begin{equation}
    \label{eq:model_weak_solution}
    \begin{aligned}
    &\int_\Omega x(t,l)\varphi(t,l)\,dl
    -
    \int_\Omega x_0(l)\varphi(0,l)\,dl
    -
    \int_0^t p(s)\varphi(s,l_0)\,ds
    \\
    &\qquad =
    \int_0^t\!\!\int_\Omega
    x(s,l)
    \big[
        \partial_s\varphi
        +
        g(E(s),l)\partial_l\varphi
        -
        (\mu(E(s),l)+u(s,l))\varphi
    \big]dl\,ds .
    \end{aligned}
\end{equation}
A global weak solution is \(x\in C^0(\mathbb R_+;L^1_+(\Omega))\) whose
restriction to each finite interval is a weak solution. For \(w\in\mathcal W\)
write \(\phi(t,x_0;w):=x(t,\cdot;x_0,w)\), and for \(\tau\ge0\) the shifted
input \(\delta_\tau w:=(\delta_\tau u,\delta_\tau p)\),
\((\delta_\tau u)(s,l):=u(\tau+s,l)\), \((\delta_\tau p)(s):=p(\tau+s)\), so
\(\delta_\tau\mathcal W\subset\mathcal W\) and
\(\phi:\mathbb R_+\times X\times\mathcal W\to X\). The control-system identities
proved later are \(\phi(0,x_0;w)=x_0\), continuity of
\(t\mapsto\phi(t,x_0;w)\), the cocycle
\(\phi(t+\tau,x_0;w)=\phi(t,\phi(\tau,x_0;w);\delta_\tau w)\), and causality
\(w|_{[0,t]}=\widetilde w|_{[0,t]}\Rightarrow\phi(t,x_0;w)=\phi(t,x_0;\widetilde w)\).
\subsection{Stationary profiles and the closure map}
\label{subsec:model_stationary_closure}
For stationary inputs \(u(t,l)=u(l)\), \(p(t)=p>0\), \(0\le u\le u_{\max}\), and
frozen \(E\ge0\), set \(g_E(l):=g(E,l)\), \(\mu_E(l):=\mu(E,l)\),
\(h_E(l):=(\mu_E(l)+u(l))/g_E(l)\). The frozen stationary equation
\[
    \frac{d}{dl}\big(g_E(l)x_E(l)\big)
    =
    -\big(\mu_E(l)+u(l)\big)x_E(l),
    \qquad
    g_E(l_0)x_E(l_0)=p
\]
gives, for the stationary flux \(y_E:=g_Ex_E\), \(y_E'=-h_Ey_E\),
\(y_E(l_0)=p\), hence
\begin{equation}
    \label{eq:model_xE}
    y_E(l)=p\,e^{-\int_{l_0}^l h_E},
    \qquad
    x_E(l)
    =
    \frac{p}{g(E,l)}
    \exp\!\left[
        -\int_{l_0}^{l}
        \frac{\mu(E,\xi)+u(\xi)}{g(E,\xi)}\,d\xi
    \right],
\end{equation}
with \(0<x_E\le p/g_m\) a.e. and \(g_Ex_E\in W^{1,\infty}(\Omega)\). The scalar
closure map is \(\Phi_u(E):=\int_\Omega\chi(l)x_E(l)\,dl\), the closed-loop
equation \(E=\Phi_u(E)\), \(H_u(E):=E-\Phi_u(E)\). A stationary equilibrium is
\((x_u^*,E_u^*)=(x_{E_u^*,u},E_u^*)\), \(H_u(E_u^*)=0\). Since \(\chi\not\equiv0\)
and \(x_E>0\), one has \(\Phi_u(0)>0\), so \(H_u(0)=-\Phi_u(0)<0\) while
\(H_u(E)\to+\infty\); thus a root \(E_u^*\ge0\) exists. With
\(\alpha_u(E,l):=-\partial_E\log g(E,l)\) and
\(b_u(E,l):=\partial_E((\mu+u)/g)
=\partial_E\mu/g-(\mu+u)\partial_Eg/g^2\),
\[
    \partial_E x_E(l)
    =
    x_E(l)
    \left[
        \alpha_u(E,l)
        -
        \int_{l_0}^{l}b_u(E,\xi)\,d\xi
    \right],
\]
\[
    \Phi_u'(E)
    =
    \int_\Omega\chi(l)\partial_E x_E(l)\,dl
    =
    \mathcal R_u(E)-\mathcal C_u(E),
\]
with \(\mathcal R_u(E):=\int_\Omega\chi x_E\alpha_u\,dl\) and
\(\mathcal C_u(E):=\int_{l_0}^{l_m}b_u(E,\xi)
(\int_{\xi}^{l_m}\chi x_E\,dl)\,d\xi\). The closure margin is
\(H_u'(E)=1-\Phi_u'(E)\).
\subsection{Stationary optimization problem}
\label{subsec:model_optimization}
Fix \(M_u>0\), \(u_{\max}>0\), and set
\[
    \mathcal U_{\rm stat}
    :=
    \{u\in BV(\Omega):0\le u\le u_{\max}\ \text{a.e.},\
    \operatorname{TV}(u)\le M_u\},
\]
sequentially compact in \(L^1(\Omega)\).
\begin{definition}[Stationary sheet]
\label{def:model_stationary_sheet}
A stationary sheet over \(\mathcal U_{\rm stat}\) is a map
\(u\mapsto E_u^*\in\mathbb R_+\) with \(E_u^*=\Phi_u(E_u^*)\); the associated
profile is \(x_u^*(l):=x_{E_u^*,u}(l)\).
\end{definition}
\begin{assumption}[Uniformly nonresonant sheet]
\label{ass:model_nonresonant_sheet}
There exist \(E_{\max},\gamma_{\rm cl},\delta_{\rm cl}>0\) and
\(s_{\rm cl}\in\{+1,-1\}\), and a stationary sheet \(u\mapsto E_u^*\), such that
\[
    0\le E_u^*\le E_{\max},\qquad u\in\mathcal U_{\rm stat},
\]
\begin{equation}
    \label{eq:model_sheet_margin}
    s_{\rm cl}\big(1-\partial_E\Phi_u(E)\big)
    \ge
    \gamma_{\rm cl},
    \qquad
    E\in(E_u^*-\delta_{\rm cl},E_u^*+\delta_{\rm cl})\cap\mathbb R_+,
\end{equation}
and \(E_u^*\) is the only zero of \(H_u\) in
\((E_u^*-\delta_{\rm cl},E_u^*+\delta_{\rm cl})\cap\mathbb R_+\).
\end{assumption}
Let \(\pi\in L^\infty_+(\Omega)\) be the unit harvesting value; the stationary
yield is \(\mathcal Y(u):=\int_\Omega\pi u x_u^*\,dl\), and the sheet-restricted
problem is
\[
    \max_{u\in\mathcal U_{\rm stat}}
    \mathcal Y(u)
    =
    \max_{u\in\mathcal U_{\rm stat}}
    \int_\Omega \pi(l)u(l)x_{E_u^*,u}(l)\,dl
\]
subject to \(E_u^*=\Phi_u(E_u^*)\), \(x_u^*=x_{E_u^*,u}\),
\(0\le u\le u_{\max}\), \(\operatorname{TV}(u)\le M_u\). For the first-order
system, \(g^*(l):=g(E_u^*,l)\), \(\mu^*(l):=\mu(E_u^*,l)\),
\(x^*(l):=x_u^*(l)\), \(a(l):=\partial_Eg(E_u^*,l)x^*(l)\),
\(m(l):=\partial_E\mu(E_u^*,l)x^*(l)\), the zero-discount adjoint operator
\[
    \mathcal L_0^*\lambda
    :=
    -g^*(l)\lambda'(l)
    +
    \big(\mu^*(l)+u(l)\big)\lambda(l),
    \qquad
    \lambda(l_m)=0,
\]
and the co-load
\[
    \langle\sigma,\lambda\rangle
    :=
    \int_\Omega
    \big(a(l)\lambda'(l)-m(l)\lambda(l)\big)\,dl .
\]
With \(R_0:=(\mathcal L_0^*)^{-1}\), \(\psi_0:=R_0\chi\), \(q_{\rm red}:=\pi u\),
\(\lambda_{\rm red}:=R_0q_{\rm red}\),
\(A_0:=\langle\sigma,\lambda_{\rm red}\rangle\),
\(B_0:=\langle\sigma,\psi_0\rangle\), \(\Gamma_0:=A_0/(1-B_0)\), the switching
function is \(S:=\pi-\lambda=S_{\rm red}-\Gamma_0\psi_0\),
\(S_{\rm red}:=\pi-\lambda_{\rm red}\).

\section{Results}
\label{sec:main_results}
\subsection{Well-posedness and positive control-system property}
\label{subsec:main_wellposedness}
Write \(E_x(t):=\langle\chi,x(t)\rangle\), \(\phi(t,x_0;w):=x(t,\cdot;x_0,w)\).
\begin{theorem}[Global well-posedness]
\label{thm:main_global_wellposedness}
Under \Cref{ass:model_coefficients,ass:model_chi}, for every
\(w=(u,p)\in\mathcal W\) and \(x_0\in X\) there is a unique
\(x(\cdot;x_0,w)\in C^0(\mathbb R_+;L^1_+(\Omega))\) satisfying
\eqref{eq:model_weak_solution}, with \(x\ge0\),
\(E=\langle\chi,x\rangle\ge0\), and for every \(T>0\),
\[
    \|x(t)\|_{L^1}
    \le
    \|x_0\|_{L^1}+p_{\max}T,
    \qquad
    0\le E(t)
    \le
    \|\chi\|_{L^\infty}\big(\|x_0\|_{L^1}+p_{\max}T\big),
\]
for \(0\le t\le T\). If \(\|x_0\|_{L^1},\|\widetilde x_0\|_{L^1}\le R\) and the
same \(w\) drives both solutions, then
\begin{equation}
    \label{eq:main_continuous_dependence}
    \|\phi(t,x_0;w)-\phi(t,\widetilde x_0;w)\|_{L^1}
    \le
    C_T(R)\|x_0-\widetilde x_0\|_{L^1},
    \qquad
    0\le t\le T .
\end{equation}
\end{theorem}
\begin{theorem}[Forward-complete positive control system]
\label{thm:main_positive_control_system}
Under the assumptions of \Cref{thm:main_global_wellposedness},
\(\phi:\mathbb R_+\times X\times\mathcal W\to X\) is a forward-complete
positive control system: \(\phi(0,x_0;w)=x_0\), \(\phi(t,x_0;w)\in X\),
\(t\mapsto\phi(t,x_0;w)\in C^0(\mathbb R_+;X)\), the Lipschitz bound
\(d_X(\phi(t,x_0;w),\phi(t,\widetilde x_0;w))\le C_T(R)d_X(x_0,\widetilde x_0)\),
causality \(w|_{[0,t]}=\widetilde w|_{[0,t]}\Rightarrow
\phi(t,x_0;w)=\phi(t,x_0;\widetilde w)\), and the cocycle
\begin{equation}
    \label{eq:main_cocycle}
    \phi(t+\tau,x_0;w)
    =
    \phi\!\left(t,\phi(\tau,x_0;w);\delta_\tau w\right),
    \qquad
    t,\tau\ge0 .
\end{equation}
\end{theorem}
\subsection{Scalar closure, derivative formula, and fold resonance}
\label{subsec:main_closure}
For stationary \(u\), \(p>0\), frozen \(E\ge0\), recall
\(h_u(E,l)=(\mu(E,l)+u(l))/g(E,l)\),
\(x_E(l)=\tfrac{p}{g(E,l)}e^{-\int_{l_0}^l h_u(E,\cdot)}\),
\(\Phi_u(E)=\int_\Omega\chi x_E\,dl\), \(H_u(E)=E-\Phi_u(E)\),
\(\alpha_u(E,l)=-\partial_E\log g(E,l)\), \(b_u(E,l)=\partial_Eh_u(E,l)\).
\begin{proposition}[Stationary profile and closure]
\label{prop:main_stationary_profile_closure}
For every \(E\ge0\), \(0\le u\le u_{\max}\), \(p>0\), the frozen stationary
problem has the unique nonnegative solution \eqref{eq:model_xE}, with
\(0<x_E\le p/g_m\) and \(g(E,\cdot)x_E\in W^{1,\infty}(\Omega)\); and
\(E^*=\Phi_u(E^*)\) iff \((x_{E^*},E^*)\) is stationary.
\end{proposition}
\begin{theorem}[Closure derivative]
\label{thm:main_closure_derivative}
For each fixed \(u\in\mathcal U_{\rm stat}\), the map
$E\mapsto \Phi_u(E)$ is \(C^1\) on bounded intervals,
\(\partial_E x_E(l)=x_E(l)[\alpha_u(E,l)-\int_{l_0}^l b_u(E,\xi)\,d\xi]\), and
  $\displaystyle \Phi_u'(E)=\mathcal R_u(E)-\mathcal C_u(E)$
with \(\mathcal R_u(E)=\int_\Omega\chi x_E\alpha_u\,dl\),
\(\mathcal C_u(E)=\int_{l_0}^{l_m}b_u(E,\xi)
(\int_{\xi}^{l_m}\chi x_E\,dl)\,d\xi\).
\end{theorem}
\begin{definition}[Closure resonance]
\label{def:main_closure_resonance}
For \(H_u(E^*)=0\), the level \(E^*\) is nonresonant iff
\(H_u'(E^*)=1-\Phi_u'(E^*)\ne0\), resonant iff \(\Phi_u'(E^*)=1\).
\end{definition}
\begin{theorem}[Implicit branch and fold resonance]
\label{thm:main_fold_resonance}
Let \(H(E,\eta)=E-\Phi(E,\eta)\), \(H(E^*,\eta^*)=0\). If
\(\partial_EH(E^*,\eta^*)\ne0\), there is a local \(C^2\) branch
\(\eta\mapsto E(\eta)\) with \(H(E(\eta),\eta)=0\) and
\[
    E'(\eta)
    =
    -\frac{\partial_\eta H}{\partial_EH}
    =
    \frac{\partial_\eta\Phi}{1-\partial_E\Phi} .
\]
If instead \(H(E^*,\eta^*)=0\), \(\partial_EH(E^*,\eta^*)=0\),
\(\partial_{EE}H(E^*,\eta^*)\ne0\), \(\partial_\eta H(E^*,\eta^*)\ne0\), then
\((E^*,\eta^*)\) is a nondegenerate fold,
\begin{equation}
    \label{eq:main_fold_branches}
    E_\pm(\eta)
    =
    E^*
    \pm
    \left[
        -\frac{2\partial_\eta H(E^*,\eta^*)}
        {\partial_{EE}H(E^*,\eta^*)}
        (\eta-\eta^*)
    \right]^{1/2}
    +
    o(|\eta-\eta^*|^{1/2}),
\end{equation}
whenever the radicand is positive.
\end{theorem}

As geometrically interpreted in \Cref{fig:scalar_closure}, the closure resonance destroys the local graph structure. This translates the analytic condition of the Sherman-Morrison singularity directly into a classic bifurcation diagram, providing a clear visual anchor for the central identity.

\subsection{Existence of an optimal policy on a nonresonant sheet}
\label{subsec:main_existence}
Let \(u\mapsto E_u^*\), \(E_u^*=\Phi_u(E_u^*)\), satisfy
\Cref{ass:model_nonresonant_sheet}; \(x_u^*:=x_{E_u^*,u}\),
\(\mathcal Y(u):=\int_\Omega\pi u x_u^*\,dl\), \(\pi\in L^\infty_+(\Omega)\).
\begin{theorem}[Continuity of the nonresonant sheet]
\label{thm:main_sheet_continuity}
If \(u_n\to u\) in \(L^1(\Omega)\) with \(u_n,u\in\mathcal U_{\rm stat}\), then
\(E_{u_n}^*\to E_u^*\) and \(x_{u_n}^*\to x_u^*\) in \(L^1(\Omega)\); more
precisely, for \(u_1,u_2\) close in \(L^1\),
\begin{equation}
    \label{eq:main_sheet_lipschitz}
    |E_{u_1}^*-E_{u_2}^*|
    \le
    \frac{C_\Phi}{\gamma_{\rm cl}g_m}
    \|u_1-u_2\|_{L^1(\Omega)}.
\end{equation}
\end{theorem}
\begin{theorem}[Existence of an optimal stationary policy]
\label{thm:main_existence_optimal_policy}
Under \Cref{ass:model_coefficients,ass:model_chi,ass:model_nonresonant_sheet},
\(\max_{u\in\mathcal U_{\rm stat}}\mathcal Y(u)\) admits a maximizer
\(u^*\in\mathcal U_{\rm stat}\).
\end{theorem}
\subsection{Rank-one adjoint and forward--adjoint identity}
\label{subsec:main_adjoint}
Fix a nonresonant stationary pair \((x^*,E^*)=(x_{E^*,u},E^*)\),
\(E^*=\Phi_u(E^*)\); \(g^*(l):=g(E^*,l)\), \(\mu^*(l):=\mu(E^*,l)\),
\(q_{\rm red}(l):=\pi(l)u(l)\),
\(\mathcal L_0^*\lambda:=-g^*\lambda'+(\mu^*+u)\lambda\), \(\lambda(l_m)=0\),
\(X_{\rm adj}:=\{\lambda\in W^{1,1}(\Omega)\cap L^\infty(\Omega):\lambda(l_m)=0\}\).
With \(a:=\partial_Eg(E^*,\cdot)x^*\), \(m:=\partial_E\mu(E^*,\cdot)x^*\),
$\displaystyle    \langle\sigma,\lambda\rangle
    :=
    \int_\Omega\big(a\lambda'-m\lambda\big)\,dl$,
and if \(a\in W^{1,1}(\Omega)\) then \(\sigma=-a'-m-a(l_0)\delta_{l_0}\).
\begin{theorem}[Lagrangian adjoint]
\label{thm:main_lagrangian_adjoint}
Let \(u^*\) be a local maximizer with state \((x^*,E^*)\) satisfying closure
nonresonance \(1-\Phi_u'(E^*)\ne0\), the inactive-\(BV\) condition
\(\operatorname{TV}(u^*)<M_u\), and the constraint qualification that the
linearized state--flux--closure map is surjective onto
\(X_{\rm adj}^*\times\mathbb R\times\mathbb R\). Then there exist multipliers
\(\lambda\in X_{\rm adj}\), \(\nu,\beta\in\mathbb R\) with
\(\nu=\lambda(l_0)\), \(\beta=\langle\sigma,\lambda\rangle\), and
\begin{equation}
    \label{eq:main_rank_one_adjoint}
    \big(\mathcal L_0^*-\chi\langle\sigma,\cdot\rangle\big)\lambda
    =
    q_{\rm red},
    \qquad
    \lambda(l_m)=0 ,
\end{equation}
together with the control variational inequality
\begin{equation}
    \label{eq:main_control_VI}
    \int_\Omega x^*(l)(\pi(l)-\lambda(l))(v(l)-u^*(l))\,dl
    \le0,
    \qquad
    0\le v\le u_{\max}.
\end{equation}
\end{theorem}
With \(R_0:=(\mathcal L_0^*)^{-1}\), \(\psi_0:=R_0\chi\),
\(\lambda_{\rm red}:=R_0q_{\rm red}\),
\(A_0:=\langle\sigma,\lambda_{\rm red}\rangle\),
\(B_0:=\langle\sigma,\psi_0\rangle\), \(\Delta_0:=1-B_0\):
\begin{theorem}[Sherman--Morrison reduction]
\label{thm:main_sherman_morrison}
If \(1-B_0\ne0\), the rank-one equation \eqref{eq:main_rank_one_adjoint} has the
unique solution
\begin{equation}
    \label{eq:main_SM_solution}
    \lambda
    =
    \lambda_{\rm red}
    +
    \frac{A_0}{1-B_0}\psi_0.
\end{equation}
If \(1-B_0=0\), then (provided \(\psi_0\not\equiv0\))
\(\psi_0\in\ker(\mathcal L_0^*-\chi\langle\sigma,\cdot\rangle)\), and the
inhomogeneous equation is solvable iff \(A_0=0\).
\end{theorem}
Let \(y:=\partial_E x_E|_{E=E^*}\), so
\(\Phi_u'(E^*)=\int_\Omega\chi y\,dl\), and
\(g^*y+a=-g^*x^*\int_{l_0}^{l}\partial_E((\mu(E^*,\xi)+u(\xi))/g^*(\xi))\,d\xi\),
whence \((g^*y+a)(l_0)=0\).
\begin{theorem}[Forward--adjoint identity]
\label{thm:main_forward_adjoint_identity}
With \(\psi_0=R_0\chi\), \(\mathcal L_0^*\psi_0=\chi\), \(\psi_0(l_m)=0\),
\[
    B_0
    =
    \langle\sigma,\psi_0\rangle
    =
    \Phi_u'(E^*),
\]
and consequently \(1-B_0=0\iff1-\Phi_u'(E^*)=0\).
\end{theorem}
\begin{corollary}[Auxiliary discount continuation]
\label{cor:main_discount_continuation}
For \(\mathcal L_r^*\lambda:=-g^*\lambda'+(r+\mu^*+u)\lambda\),
\(R_r:=(\mathcal L_r^*)^{-1}\), \(B_r:=\langle\sigma,R_r\chi\rangle\), one has
\(B_r=B_0+\mathcal O(r)=\Phi_u'(E^*)+\mathcal O(r)\) as \(r\downarrow0\). The optimality system
for \(\mathcal Y\) uses \(r=0\); \(r>0\) is only a resolvent-continuation
parameter.
\end{corollary}

To elucidate the functional analytic framework, \Cref{fig:adjoint_loop} maps the abstract Sherman-Morrison reduction established in Theorem~\ref{thm:main_sherman_morrison} into an equivalent control-theoretic block diagram. This signal-flow representation characterizes the exact alignment between the rank-one forward-adjoint identity and the closed-loop system structure.

\subsection{Switching geometry and diagnostics}
\label{subsec:main_switching}
With \(S:=\pi-\lambda\), \(S_{\rm red}:=\pi-\lambda_{\rm red}\),
\(\Gamma_0:=A_0/(1-B_0)\), \eqref{eq:main_SM_solution} gives
\[
S(l)=S_{\rm red}(l)-\Gamma_0\psi_0(l).
\]
\begin{theorem}[Pointwise maximum condition and bang--bang law]
\label{thm:main_bang_bang}
At a regular local maximizer with inactive \(BV\)-constraint,
\(u^*(l)\in\operatorname*{arg\,max}_{0\le v\le u_{\max}}v\,x^*(l)S(l)\) a.e., and
since \(x^*>0\),
\begin{equation}
    \label{eq:main_bang_bang}
    u^*(l)=
    \begin{cases}
        u_{\max}, & S(l)>0,\\[1mm]
        0, & S(l)<0,
    \end{cases}
    \qquad
    \text{a.e.}
\end{equation}
\end{theorem}
Let \(S_{\rm red}\in C^1(\Omega)\) have a unique simple zero
\(l_{\rm red}\in(l_0,l_m)\), \(S_{\rm red}'(l_{\rm red})\ne0\). For
\(\delta>0\), \(\mathcal N_\delta:=(l_{\rm red}-\delta,l_{\rm red}+\delta)\cap\Omega\),
\(m_{\rm out}:=\inf_{\Omega\setminus\mathcal N_\delta}|S_{\rm red}|\),
\(m_1:=\inf_{\mathcal N_\delta}|S_{\rm red}'|\),
\(\rho_\psi:=|\Gamma_0|\|\psi_0\|_{L^\infty}\),
\(\rho_{\psi,1}:=|\Gamma_0|\|\psi_0'\|_{L^\infty(\mathcal N_\delta)}\).
\begin{theorem}[Single-threshold preservation]
\label{thm:main_single_threshold}
If \(\rho_\psi<m_{\rm out}\) and \(\rho_{\psi,1}<m_1\), then \(S\) has exactly
one zero \(l_*\in\Omega\), and
\begin{equation}
    \label{eq:main_threshold_shift}
    l_*-l_{\rm red}
    =
    \frac{\Gamma_0\psi_0(l_{\rm red})}
    {S_{\rm red}'(l_{\rm red})}
    +
    o(|\Gamma_0|),
    \qquad
    \Gamma_0\to0.
\end{equation}
\end{theorem}

The robustness of the harvesting strategy under small perturbations is geometrically verified in \Cref{fig:single_threshold}. As demonstrated, the perturbed switching function $S(l)$ remains strictly bounded within the perturbation envelope defined by $\rho_\psi$. Driven by the transversal crossing condition $m_{\rm out}$ of the unperturbed signal $S_{\rm red}(l)$, the environmental feedback merely shifts the threshold from $l_{\rm red}$ to $l_*$ without creating spurious switching points, validating the single-threshold preservation established in Theorem~\ref{thm:main_single_threshold}.

\begin{theorem}[Tangency birth of harvest windows]
\label{thm:main_tangency}
Let \(S_{\rm red},\psi_0\in C^2(\Omega)\). A double switching point satisfies
\(S(l_t)=S'(l_t)=0\), equivalently
\(S_{\rm red}(l_t)=\Gamma_{\rm tan}\psi_0(l_t)\),
\(S_{\rm red}'(l_t)=\Gamma_{\rm tan}\psi_0'(l_t)\). If
\(\psi_0(l_t)\psi_0'(l_t)\ne0\), then
\begin{equation}
    \label{eq:main_tangency_eliminated}
    S_{\rm red}(l_t)\psi_0'(l_t)
    -
    S_{\rm red}'(l_t)\psi_0(l_t)=0,
    \qquad
    \Gamma_{\rm tan}
    =
    \frac{S_{\rm red}(l_t)}{\psi_0(l_t)}.
\end{equation}
If additionally \(S_{\rm red}''(l_t)-\Gamma_{\rm tan}\psi_0''(l_t)\ne0\) and
\(\psi_0(l_t)\ne0\), then \(F(l,\Gamma):=S_{\rm red}(l)-\Gamma\psi_0(l)\) has a
nondegenerate fold,
\begin{equation}
    \label{eq:main_tangency_branches}
    l_\pm(\Gamma)
    =
    l_t
    \pm
    \left[
        \frac{2\psi_0(l_t)(\Gamma-\Gamma_{\rm tan})}
        {S_{\rm red}''(l_t)-\Gamma_{\rm tan}\psi_0''(l_t)}
    \right]^{1/2}
    +
    o(|\Gamma-\Gamma_{\rm tan}|^{1/2}),
\end{equation}
and if \(S>0\) on \((l_-(\Gamma),l_+(\Gamma))\), then \(u^*=u_{\max}\) there:
the two switches bound a harvest window.
\end{theorem}

Let \(\widetilde E\) be a computed closure level,
\(\rho_{\rm cl}:=|\widetilde E-\Phi_h(\widetilde E)|\).
\begin{proposition}[Closure residual certificate]
\label{prop:main_closure_certificate}
If \(H(E)=E-\Phi(E)\), \(H(E^*)=0\), \(H'\ge m>0\) on \(I\), and
\(E^*,\widetilde E\in I\), then \(|E^*-\widetilde E|\le\rho_{\rm cl}/m\).
\end{proposition}
Let \(B_h(0):=\langle\sigma_h,\psi_{0,h}\rangle_h\) and \(\Phi_h'(\widetilde E)\)
be the discrete closure derivative, and
\(\mathcal D_{\rm FA}^h:=B_h(0)-\Phi_h'(\widetilde E)\).
\begin{proposition}[Forward--adjoint diagnostic]
\label{prop:main_FA_diagnostic}
If \(\sigma_h\to\sigma\), \(\psi_{0,h}\to\psi_0\),
\(\Phi_h'(\widetilde E)\to\Phi_u'(E^*)\), then
\(\mathcal D_{\rm FA}^h\to B_0-\Phi_u'(E^*)=0\). If the boundary correction is
omitted and \(a\in W^{1,1}\), then formally
\(\mathcal D_{\rm FA}^{h,\rm int}\approx a_h(l_0)\psi_{0,h}(l_0)\), so a
persistent nonzero residual of this size signals an omitted or mis-signed
fixed-flux boundary correction.
\end{proposition}
A numerical report should include \(\rho_{\rm cl}\),
\(m_h:=\inf_{I_h}|1-\Phi_h'|\), \(\mathcal D_{\rm FA}^h\), \(1-B_h(0)\),
\(\Gamma_h:=A_h(0)/(1-B_h(0))\), and \(\#\{l:S_h(l)=0\}\).

As illustrated in \Cref{fig:tangency_birth}, the topological creation of a new harvesting window is strictly governed by the fold condition in the switching function. This geometric perspective confirms the bifurcation structure detailed in Theorem~\ref{thm:main_tangency}, showing exactly how varying the parameter $\Gamma$ transitions the system from a strict no-harvest regime into one featuring an interior bang-bang harvesting window.

\section{Results (Proofs)}
\label{sec:proofs}
The frozen-path estimate
\(\|z_{e_1}(t)-z_{e_2}(t)\|_{L^1}\le C_T\int_0^t|e_1-e_2|\) is proved in
Appendix~\ref{app:frozen_flow_stability} and used only in
\Cref{subsec:proof_wellposedness}.
\subsection{Proof of well-posedness and control-system property}
\label{subsec:proof_wellposedness}
\begin{proof}[Proof of \Cref{thm:main_global_wellposedness}]
Fix \(T>0\) and a frozen path \(e\in C^0([0,T];[0,E_T])\); set
\(g_e(t,l):=g(e(t),l)\), \(\mu_e(t,l):=\mu(e(t),l)\). The frozen problem
\(\partial_t z+\partial_l(g_ez)=-(\mu_e+u)z\),
\(g_e(t,l_0)z(t,l_0)=p(t)\), \(z(0)=x_0\) is solved by characteristics: the
backward flow \(\frac{d}{ds}L_e(s;t,l)=g_e(s,L_e)\), \(L_e(t;t,l)=l\), reaches
either \(s=0\) or \(l=l_0\) since \(g_e\ge g_m>0\). With
\(\ell_e(t):=L_e(t;0,l_0)\) and boundary entrance time \(\tau_b\), the
nonconservative form \(\partial_t z+g_e\partial_lz=-(\partial_lg_e+\mu_e+u)z\)
gives
\[
    z(t,l)
    =
    x_0(L_e(0;t,l))
    \exp\!\Big[-\!\int_0^t(\partial_lg_e+\mu_e+u)(s,L_e(s;t,l))\,ds\Big],
    \quad l>\ell_e(t),
\]
\[
    z(t,l)
    =
    \frac{p(\tau_b)}{g_e(\tau_b,l_0)}
    \exp\!\Big[-\!\int_{\tau_b}^t(\partial_lg_e+\mu_e+u)(s,L_e(s;t,l))\,ds\Big],
    \quad l<\ell_e(t).
\]
Since \(x_0\ge0\), \(p\ge0\), and the exponentials are positive, \(z_e\ge0\).
Define \(\mathcal T e(t):=\int_\Omega\chi z_e(t)\,dl\). With
\(M_T:=\|x_0\|_{L^1}+p_{\max}T\), \(E_T:=\|\chi\|_{L^\infty}M_T\), the trace-free
mass estimate gives \(\|z_e(t)\|_{L^1}\le M_T\), hence
\(0\le\mathcal T e(t)\le E_T\).

Moreover, since \(g_e\le C_{E_T}\) (finite propagation speed) and the data
\(x_0,u(t,\cdot)\in BV\) with \(p\in BV_{\rm loc}\), the map
\(t\mapsto z_e(t)\) is Lipschitz into \(L^1(\Omega)\). Indeed, the
characteristic formulae
\eqref{eq:app_initial_branch}--\eqref{eq:app_boundary_branch} express
\(z_e(t,\cdot)\) as compositions of \(BV\) data with uniformly bi-Lipschitz
characteristic maps (whose images move by \(\mathcal O(|t-s|)\) since \(g_e\le C_{E_T}\)),
multiplied by uniformly Lipschitz survival factors, while the moving branch
interface \(\ell_e\) contributes only a strip of measure \(\mathcal O(|t-s|)\); the
\(BV\)-composition estimate of \Cref{lem:app_bv_composition} then gives
\(\|z_e(t)-z_e(s)\|_{L^1}\le C_T|t-s|\). Consequently
\(t\mapsto\mathcal T e(t)=\int_\Omega\chi z_e(t)\) is Lipschitz with a constant
\(\Lambda_T\) depending only on the data bounds, uniformly in \(e\). Hence
\(\mathcal T\) maps $\mathcal K_T:=\{e\in C^0([0,T];[0,E_T]):\operatorname{Lip}(e)\le\Lambda_T\}$
into itself. Since a uniform limit of functions with Lipschitz constant at most
\(\Lambda_T\) and range in the closed interval \([0,E_T]\) again has Lipschitz
constant at most \(\Lambda_T\) and range in \([0,E_T]\), \(\mathcal K_T\) is
closed in \(C^0([0,T])\), hence a complete metric space. By
Appendix~\ref{app:frozen_flow_stability},
\(\|z_{e_1}(t)-z_{e_2}(t)\|_{L^1}\le C_T\int_0^t|e_1-e_2|\), so
\(|\mathcal T e_1(t)-\mathcal T e_2(t)|\le\|\chi\|_{L^\infty}C_T\int_0^t|e_1-e_2|\)
and
\(\|\mathcal T e_1-\mathcal T e_2\|_{C^0([0,T])}\le\|\chi\|_{L^\infty}C_TT
\|e_1-e_2\|_{C^0([0,T])}\). Choosing \(T_0\) with
\(\|\chi\|_{L^\infty}C_{T_0}T_0<1\), \(\mathcal T\) is a contraction on the
complete metric space \(\mathcal K_{T_0}\), giving a unique fixed point
\(E=\mathcal T E\); set \(x:=z_E\), so \(E=\langle\chi,x\rangle\) and \(x\)
solves \eqref{eq:model_weak_solution}.

The a priori bounds \(\|x(t)\|_{L^1}\le\|x_0\|_{L^1}+p_{\max}T\) and
\(0\le E(t)\le\|\chi\|_{L^\infty}(\|x_0\|_{L^1}+p_{\max}T)\) preclude
finite-time blow-up, so \(x\in C^0(\mathbb R_+;L^1_+(\Omega))\). For continuous
dependence with common \(w\),
\(|E(t)-\widetilde E(t)|\le\|\chi\|_{L^\infty}\|x(t)-\widetilde x(t)\|_{L^1}\),
and applying the estimate \eqref{eq:app_frozen_path_stability} with \(e_1=E\), \(e_2=\widetilde E\) (both
Lipschitz, as fixed points lie in \(\mathcal K_T\)),
\[
    \|x(t)-\widetilde x(t)\|_{L^1}
    \le
    C_T\|x_0-\widetilde x_0\|_{L^1}
    +
    C_T\int_0^t\|x(s)-\widetilde x(s)\|_{L^1}\,ds,
\]
and Gronwall gives \eqref{eq:main_continuous_dependence} with
\(C_T(R)=C_Te^{C_TT}\); taking \(x_0=\widetilde x_0\) gives uniqueness.
\end{proof}
\begin{proof}[Proof of \Cref{thm:main_positive_control_system}]
From \Cref{thm:main_global_wellposedness}, \(\phi(t,x_0;w)=x(t,\cdot;x_0,w)\) is
defined for all \(t\ge0\), \(x_0\in X\), \(w\in\mathcal W\), with
\(\phi(0,x_0;w)=x_0\), \(\phi(t,x_0;w)\in X\), continuity in \(t\), and the
Lipschitz bound. If \(w|_{[0,t]}=\widetilde w|_{[0,t]}\), both trajectories
solve the same weak problem on \([0,t]\), so
\(\phi(t,x_0;w)=\phi(t,x_0;\widetilde w)\). For the cocycle, put
\(y(s):=\phi(s+\tau,x_0;w)\); the change of variables \(r=s+\tau\) in
\eqref{eq:model_weak_solution} gives \(y(0)=\phi(\tau,x_0;w)\) and
\(y(s)=\phi(s,\phi(\tau,x_0;w);\delta_\tau w)\), i.e. \eqref{eq:main_cocycle}.
\end{proof}
\subsection{Proofs for the closure map and fold resonance}
\label{subsec:proof_closure}
\begin{proof}[Proof of \Cref{prop:main_stationary_profile_closure}]
For \(y_E:=g(E,\cdot)x_E\), \(y_E'=-\tfrac{\mu(E,l)+u(l)}{g(E,l)}y_E\),
\(y_E(l_0)=p\), so \(y_E=p\,e^{-\int_{l_0}^l(\mu(E,\cdot)+u)/g(E,\cdot)}\) and
\(x_E=y_E/g(E,\cdot)\) is \eqref{eq:model_xE}. Since
\(0\le(\mu+u)/g\le(\|\mu(E,\cdot)\|_\infty+u_{\max})/g_m\),
\[
    0<
    \frac{p}{\|g(E,\cdot)\|_\infty}
    e^{-(\|\mu(E,\cdot)\|_\infty+u_{\max})|\Omega|/g_m}
    \le
    x_E
    \le
    \frac{p}{g_m}.
\]
As \(h_E\in L^\infty\), \(y_E=g_Ex_E\in W^{1,\infty}(\Omega)\); uniqueness is
from the scalar linear equation for \(y_E\); and \(E^*=\Phi_u(E^*)\) iff
\(E^*=\langle\chi,x_{E^*}\rangle\), the stationary equation being the frozen one
at \(E=E^*\).
\end{proof}
\begin{proof}[Proof of \Cref{thm:main_closure_derivative}]
Differentiating
\(\log x_E=\log p-\log g(E,l)-\int_{l_0}^l(\mu(E,\xi)+u(\xi))/g(E,\xi)\,d\xi\) in
\(E\) gives \(\partial_E x_E=x_E[\alpha_u-\int_{l_0}^l b_u]\). Since
\(x_E,\alpha_u,b_u\in L^\infty\) and \(|\Omega|<\infty\), dominated convergence
yields \(\Phi_u'(E)=\int_\Omega\chi\partial_E x_E\,dl\); substituting and
applying Fubini,
\[
    \int_{l_0}^{l_m}\chi x_E\Big(\int_{l_0}^l b_u\,d\xi\Big)dl
    =
    \int_{l_0}^{l_m}b_u(E,\xi)\Big(\int_\xi^{l_m}\chi x_E\,dl\Big)d\xi,
\]
so \(\Phi_u'=\mathcal R_u-\mathcal C_u\). Continuity of \(\Phi_u'\) on bounded
\(E\)-intervals follows from the \(L^\infty\)-continuity of
\(E\mapsto\partial_Eg(E,\cdot),\partial_E\mu(E,\cdot)\)
(\Cref{ass:model_coefficients}) and dominated convergence.
\end{proof}
\begin{proof}[Proof of \Cref{thm:main_fold_resonance}]
If \(\partial_EH(E^*,\eta^*)\ne0\), the implicit-function theorem gives a
\(C^2\) branch with, on differentiating \(H(E(\eta),\eta)=0\),
\(E'(\eta)=-\partial_\eta H/\partial_EH=\partial_\eta\Phi/(1-\partial_E\Phi)\).
At resonance \(\partial_EH(E^*,\eta^*)=0\); writing \(z=E-E^*\),
\(\rho=\eta-\eta^*\), Taylor expansion gives
\[
    H(E^*+z,\eta^*+\rho)
    =
    \partial_\eta H(E^*,\eta^*)\rho
    +
    \tfrac12\partial_{EE}H(E^*,\eta^*)z^2
    +
    o(|\rho|+|z|^2),
\]
so the leading balance \(\partial_\eta H\,\rho+\tfrac12\partial_{EE}H\,z^2=0\)
yields \(z_\pm=\pm[-2\partial_\eta H/\partial_{EE}H\,(\eta-\eta^*)]^{1/2}
+o(|\eta-\eta^*|^{1/2})\), i.e. \eqref{eq:main_fold_branches}.
\end{proof}
\subsection{Proof of existence on the nonresonant sheet}
\label{subsec:proof_existence}
\begin{proof}[Proof of \Cref{thm:main_sheet_continuity}]
For fixed \(E\in[0,E_{\max}]\), with
\(A_i(E,l):=\int_{l_0}^l(\mu(E,\xi)+u_i(\xi))/g(E,\xi)\,d\xi\),
\(|A_1-A_2|\le g_m^{-1}\|u_1-u_2\|_{L^1}\), and \(|e^{-A_1}-e^{-A_2}|\le|A_1-A_2|\),
\(0\le e^{-A_i}\le1\), give
\[
    |\Phi_{u_1}(E)-\Phi_{u_2}(E)|
    \le
    \frac{p\|\chi\|_{L^\infty}|\Omega|}{g_m^2}\|u_1-u_2\|_{L^1}
    =:\rho_{12}.
\]
By the sign-definite margin \eqref{eq:model_sheet_margin},
\(H_{u_2}'=1-\partial_E\Phi_{u_2}\) has constant sign \(s_{\rm cl}\) and
magnitude \(\ge\gamma_{\rm cl}\) on the isolating interval, so if \(E_{u_1}^*\)
lies in it, the mean value theorem gives
\(\gamma_{\rm cl}|E_{u_1}^*-E_{u_2}^*|\le|H_{u_2}(E_{u_1}^*)|
=|\Phi_{u_1}(E_{u_1}^*)-\Phi_{u_2}(E_{u_1}^*)|\le\rho_{12}\), i.e.
\eqref{eq:main_sheet_lipschitz} with \(C_\Phi=p\|\chi\|_{L^\infty}|\Omega|/g_m\).
For \(u_n\to u\) this confinement gives \(E_{u_n}^*\to E_u^*\); then, from
\eqref{eq:model_xE}, \(E_{u_n}^*\to E_u^*\) and \(u_n\to u\) in \(L^1\) give
\(x_{u_n}^*\to x_u^*\) a.e. with \(0\le x_{u_n}^*\le p/g_m\), so dominated
convergence yields \(\|x_{u_n}^*-x_u^*\|_{L^1}\to0\).
\end{proof}
\begin{proof}[Proof of \Cref{thm:main_existence_optimal_policy}]
Let \(\{u_n\}\subset\mathcal U_{\rm stat}\) be maximizing. The bounds
\(0\le u_n\le u_{\max}\), \(\operatorname{TV}(u_n)\le M_u\) make \(\{u_n\}\)
bounded in \(BV\); by compactness in \(L^1\), a subsequence converges to
\(u^*\in L^1\) a.e., and lower semicontinuity of total variation gives
\(u^*\in\mathcal U_{\rm stat}\). By \Cref{thm:main_sheet_continuity},
\(x_{u_n}^*\to x_{u^*}^*\) in \(L^1\), so
\[
    |\mathcal Y(u_n)-\mathcal Y(u^*)|
    \le
    \|\pi\|_\infty\frac{p}{g_m}\|u_n-u^*\|_{L^1}
    +
    \|\pi\|_\infty u_{\max}\|x_{u_n}^*-x_{u^*}^*\|_{L^1}
    \to0,
\]
whence \(u^*\) attains the supremum.
\end{proof}

\subsection{Lagrangian derivation and proof of the adjoint formula}
\label{subsec:proof_adjoint}
\begin{proof}[Proof of \Cref{thm:main_lagrangian_adjoint}]
Take \(x\) in the trace space \(\{x:g(E,\cdot)x\in W^{1,1}(\Omega)\}\), on which
the boundary values \(x(l_0),x(l_m)\) are defined, and use the constrained
Lagrangian with the transport constraint integrated by parts,
\(-\int_\Omega\lambda\,\tfrac{d}{dl}(gx)
=\int_\Omega g(E,l)x\,\lambda'-[\lambda g(E,\cdot)x]_{l_0}^{l_m}\),
\[
\begin{aligned}
    \mathfrak L
    &=
    \int_\Omega \pi ux
    +
    \int_\Omega g(E,l)x\,\lambda'
    -
    [\lambda g(E,\cdot)x]_{l_0}^{l_m}
    \\
    &\quad
    -
    \int_\Omega(\mu(E,l)+u)x\lambda
    -
    \nu(g(E,l_0)x(l_0)-p)
    -
    \beta\Big(E-\int_\Omega\chi x\Big).
\end{aligned}
\]
The \(x\)-variation is
\[
\begin{aligned}
    D_x\mathfrak L[\delta x]
    &=
    \int_\Omega(\pi u+\beta\chi)\delta x
    +
    \int_\Omega g^*\delta x\,\lambda'
    -
    \int_\Omega(\mu^*+u)\delta x\,\lambda
    \\
    &\quad
    -
    \lambda(l_m)g^*(l_m)\delta x(l_m)
    +
    \lambda(l_0)g^*(l_0)\delta x(l_0)
    -
    \nu g^*(l_0)\delta x(l_0);
\end{aligned}
\]
independence of the boundary variations forces \(\lambda(l_m)=0\) and
\(\nu=\lambda(l_0)\), and the interior coefficient
\(\pi u+\beta\chi+g^*\lambda'-(\mu^*+u)\lambda=0\), i.e.
\(\mathcal L_0^*\lambda-\beta\chi=q_{\rm red}\). For the \(E\)-variation, with
\(a=\partial_Eg(E^*,\cdot)x^*\), \(m=\partial_E\mu(E^*,\cdot)x^*\),
\[
    D_E\mathfrak L[\delta E]
    =
    \delta E
    \Big[
        \int_\Omega a\lambda'
        +
        a(l_0)\lambda(l_0)
        -
        \int_\Omega m\lambda
        -
        \nu a(l_0)
        -
        \beta
    \Big],
\]
and \(\nu=\lambda(l_0)\) cancels the boundary pair, leaving
\(\beta=\int_\Omega(a\lambda'-m\lambda)=\langle\sigma,\lambda\rangle\).
Substituting gives \eqref{eq:main_rank_one_adjoint}. Existence of the
multipliers is the Banach-space Lagrange multiplier rule under the stated
surjectivity (the transport block is invertible by \(g^*\ge g_m\); the scalar
block by \(1-\Phi_u'(E^*)\ne0\)). Finally
\(D_u\mathfrak L[\delta u]=\int_\Omega x^*(\pi-\lambda)\delta u\) gives, over the
box \(0\le v\le u_{\max}\), the variational inequality
\eqref{eq:main_control_VI}.
\end{proof}
\begin{proof}[Proof of \Cref{thm:main_sherman_morrison}]
Applying \(R_0=(\mathcal L_0^*)^{-1}\) to \eqref{eq:main_rank_one_adjoint},
\(\lambda=\lambda_{\rm red}+\psi_0\langle\sigma,\lambda\rangle\); pairing with
\(\sigma\), \((1-B_0)\langle\sigma,\lambda\rangle=A_0\). If \(1-B_0\ne0\),
\(\langle\sigma,\lambda\rangle=A_0/(1-B_0)\) and
\(\lambda=\lambda_{\rm red}+\tfrac{A_0}{1-B_0}\psi_0\). If \(1-B_0=0\), then
\((\mathcal L_0^*-\chi\langle\sigma,\cdot\rangle)\psi_0=\chi-\chi B_0=(1-B_0)\chi=0\),
and the compatibility condition is \(A_0=0\).
\end{proof}
\subsection{Proof of the forward--adjoint identity}
\label{subsec:proof_forward_adjoint}
\begin{proof}[Proof of \Cref{thm:main_forward_adjoint_identity}]
Let \(y:=\partial_E x_E|_{E=E^*}\). Differentiating
\(g(E,l)x_E=p\,e^{-\int_{l_0}^l(\mu(E,\xi)+u(\xi))/g(E,\xi)\,d\xi}\) at \(E=E^*\)
gives
\(g^*y+a=-g^*x^*\int_{l_0}^l\partial_E((\mu(E^*,\xi)+u)/g^*)\,d\xi\), so
\((g^*y+a)(l_0)=0\).
Note that
$g^*y+a\in W^{1,1}(\Omega)\cap L^\infty(\Omega)$.
Indeed, the differentiated flux equation gives
\[
    (g^*y+a)'=-(\mu^*+u)y-m
    \quad\text{in }\mathcal D'(l_0,l_m),
\]
and the right-hand side belongs to \(L^1(\Omega)\). Hence integration by parts
against \(\varphi\in X_{\rm adj}\) is justified.

For \(\varphi\in X_{\rm adj}\), testing and integrating by parts,
using \(\varphi(l_m)=0\), \((g^*y+a)(l_0)=0\),
\[
    0=
    -\int_\Omega(g^*y+a)\varphi'
    +
    \int_\Omega(\mu^*+u)y\varphi
    +
    \int_\Omega m\varphi,
\]
i.e. \(\int_\Omega y[-g^*\varphi'+(\mu^*+u)\varphi]=\int_\Omega(a\varphi'-m\varphi)\),
that is \(\int_\Omega y\,\mathcal L_0^*\varphi=\langle\sigma,\varphi\rangle\).
Taking \(\varphi=\psi_0=R_0\chi\), \(\mathcal L_0^*\psi_0=\chi\),
\[
    \Phi_u'(E^*)
    =
    \int_\Omega\chi y
    =
    \int_\Omega y\,\mathcal L_0^*\psi_0
    =
    \langle\sigma,\psi_0\rangle
    =
    B_0.
\]
\end{proof}
\begin{proof}[Proof of \Cref{cor:main_discount_continuation}]
With \(q_r:=(r+\mu^*+u)/g^*\) and
\((R_rf)(l)=\int_l^{l_m}\tfrac{f(s)}{g^*(s)}e^{-\int_l^s q_r}\,ds\), and
\(A(l,s):=\int_l^s d\xi/g^*(\xi)\in[0,|\Omega|/g_m]\),
\(|e^{-rA}-1|\le rA\) gives \(\|R_rf-R_0f\|_{L^\infty}\le Cr\|f\|_{L^1}\). Since
\((R_rf)'=-f/g^*+q_rR_rf\), the derivative component picks up
\(\|q_r\|_{L^\infty}\), so \(\|(R_rf)'-(R_0f)'\|_{L^1}\le Cr\|f\|_{L^1}\) and
\(\|R_r-R_0\|_{\mathcal L(L^1,X_{\rm adj})}\le Cr\) with
\(C=C(g_m,\|\mu^*+u\|_{L^\infty},|\Omega|)\). Hence
\(|B_r-B_0|=|\langle\sigma,(R_r-R_0)\chi\rangle|
\le\|\sigma\|_{X_{\rm adj}^*}Cr\|\chi\|_{L^1}\le Cr\), and by
\Cref{thm:main_forward_adjoint_identity},
\(B_r=\Phi_u'(E^*)+\mathcal O(r)\).
\end{proof}
\subsection{Proofs of the switching and diagnostic results}
\label{subsec:proof_switching}
\begin{proof}[Proof of \Cref{thm:main_bang_bang}]
From \Cref{thm:main_lagrangian_adjoint}, \(\int_\Omega x^*S(v-u^*)\le0\) for
\(0\le v\le u_{\max}\). The box \(\mathcal U_{\rm box}=\{0\le v\le u_{\max}\}\)
is decomposable, so the inequality localizes to
\(u^*(l)\in\operatorname*{arg\,max}_{0\le v\le u_{\max}}v\,x^*(l)S(l)\) a.e.;
since \(x^*>0\), the affine map \(v\mapsto v\,x^*(l)S(l)\) is maximized at
\(v=u_{\max}\) if \(S(l)>0\) and \(v=0\) if \(S(l)<0\), giving
\eqref{eq:main_bang_bang}.
\end{proof}
\begin{proof}[Proof of \Cref{thm:main_single_threshold}]
With \(S=S_{\rm red}-\Gamma_0\psi_0\): on \(\Omega\setminus\mathcal N_\delta\),
\(|S|\ge m_{\rm out}-\rho_\psi>0\), so no zero occurs there; on
\(\mathcal N_\delta\), \(|S'|\ge m_1-\rho_{\psi,1}>0\), so \(S\) is strictly
monotone. Since \(S_{\rm red}\) changes sign across \(l_{\rm red}\) and the
perturbation is below the outer margin, \(S\) has exactly one zero \(l_*\).
Expanding \(0=S(l_*)=S_{\rm red}(l_*)-\Gamma_0\psi_0(l_*)\) at \(l_{\rm red}\)
gives \eqref{eq:main_threshold_shift}.
\end{proof}
\begin{proof}[Proof of \Cref{thm:main_tangency}]
For \(F(l,\Gamma)=S_{\rm red}(l)-\Gamma\psi_0(l)\), a double point is
\(F(l_t,\Gamma_{\rm tan})=0\), \(\partial_lF(l_t,\Gamma_{\rm tan})=0\), i.e.
\(S_{\rm red}(l_t)=\Gamma_{\rm tan}\psi_0(l_t)\),
\(S_{\rm red}'(l_t)=\Gamma_{\rm tan}\psi_0'(l_t)\); eliminating
\(\Gamma_{\rm tan}\) gives \eqref{eq:main_tangency_eliminated}. Since
\(\partial_\Gamma F(l_t,\Gamma_{\rm tan})=-\psi_0(l_t)\ne0\) and
\(\partial_{ll}F(l_t,\Gamma_{\rm tan})=S_{\rm red}''(l_t)-\Gamma_{\rm tan}\psi_0''(l_t)\ne0\),
\(F=0\) has a fold in \((l,\Gamma)\); Taylor expansion
\[
    0=
    -\psi_0(l_t)(\Gamma-\Gamma_{\rm tan})
    +
    \tfrac12\big(S_{\rm red}''(l_t)-\Gamma_{\rm tan}\psi_0''(l_t)\big)(l-l_t)^2
    +
    o(|\Gamma-\Gamma_{\rm tan}|+|l-l_t|^2)
\]
gives \eqref{eq:main_tangency_branches}; if \(S>0\) on \((l_-,l_+)\), the
bang--bang law gives \(u^*=u_{\max}\) there.
\end{proof}
\begin{proof}[Proof of \Cref{prop:main_closure_certificate}]
By the mean value theorem
\(|H(\widetilde E)-H(E^*)|=|H'(\xi)||\widetilde E-E^*|\ge m|\widetilde E-E^*|\),
and \(H(E^*)=0\), \(|H(\widetilde E)|=\rho_{\rm cl}\) give
\(|\widetilde E-E^*|\le\rho_{\rm cl}/m\).
\end{proof}
\begin{proof}[Proof of \Cref{prop:main_FA_diagnostic}]
By consistency \(B_h(0)\to B_0\) and \(\Phi_h'(\widetilde E)\to\Phi_u'(E^*)\),
so \(\mathcal D_{\rm FA}^h\to B_0-\Phi_u'(E^*)=0\) by
\Cref{thm:main_forward_adjoint_identity}. If \(a\in W^{1,1}\), integration by
parts gives \(\langle\sigma,\psi_0\rangle=\int_\Omega(-a'-m)\psi_0-a(l_0)\psi_0(l_0)\),
so the uncorrected interior co-load \(\sigma_{\rm int}=-a'-m\) satisfies
\(\int_\Omega\sigma_{\rm int}\psi_0=\Phi_u'(E^*)+a(l_0)\psi_0(l_0)\), whence the
uncorrected residual is \(\mathcal D_{\rm FA}^{h,\rm int}\approx a_h(l_0)\psi_{0,h}(l_0)\).
\end{proof}

\section{Discussion}
\label{sec:conclusion}
For the closed-loop size-structured system
\eqref{eq:model_state}--\eqref{eq:model_flux_bc} we established a
forward-complete positive control system on \(L^1_+(\Omega)\), reduced the
stationary equilibria to the scalar closure \(E=\Phi_u(E)\) with derivative
\(\Phi_u'=\mathcal R_u-\mathcal C_u\) and a generic fold at
\(\Phi_u'(E^*)=1\), and proved existence of an optimal policy on a uniformly
nonresonant sheet. From the constrained Lagrangian we obtained the rank-one
adjoint \eqref{eq:main_rank_one_adjoint} with the boundary-corrected co-load
\(\sigma\), the Sherman--Morrison reduction \eqref{eq:main_SM_solution}, and the
central identity
$B_0=\langle\sigma,\psi_0\rangle=\Phi_u'(E^*)$,
which makes closure resonance, zero-discount adjoint resonance, and singularity
of the Sherman--Morrison denominator one and the same scalar condition. The
same margin governs the switching decomposition
\(S=S_{\rm red}-\Gamma_0\psi_0\), single-threshold preservation, and the
tangency birth of harvest windows, so that equilibrium sensitivity, adjoint
solvability, and threshold stability are controlled together.

The boundary correction is structurally necessary rather than a lower-order
detail: because recruitment is fixed as a flux, the Lagrangian forces
\(\nu=\lambda(l_0)\) and the Dirac term \(-a(l_0)\delta_{l_0}\), and the
uncorrected interior co-load \(\sigma_{\rm int}=-a'-m\) satisfies
\(\int_\Omega\sigma_{\rm int}\psi_0=\Phi_u'(E^*)+a(l_0)\psi_0(l_0)\), a bias the
diagnostic \(\mathcal D_{\rm FA}^h\) detects.

Several problems remain open. First, at a fold \(1-\Phi_u'(E^*)=0\) the map
\(H_u\) loses local invertibility and the sheet may cease to be a graph,
replacing \(u\mapsto E_u^*\) by the set
\(\mathcal E_u^*:=\{E\ge0:E=\Phi_u(E)\}\) and calling for a set-valued or
branch-selection optimality theory. Second, the genuinely discounted dynamic
problem \(\max_u\int_0^\infty e^{-rt}\int_\Omega\pi u x\,dl\,dt\), \(r>0\),
should produce a true current-value adjoint with
\(\mathcal L_r^*=-g^*\partial_l+(r+\mu^*+u)\), whereas here \(r>0\) is only a
continuation parameter with \(B_r=\Phi_u'(E^*)+\mathcal O(r)\). Third, adding diffusion,
stochastic recruitment, or multiple structured species may replace the scalar
loop \(x\mapsto\langle\chi,x\rangle\) by finite-rank, compact, or fully
nonlinear feedback,
\(\chi\langle\sigma,\cdot\rangle\rightsquigarrow\sum_{j=1}^N\chi_j\langle\sigma_j,\cdot\rangle\)
or \(\mathcal K(E,\cdot)\), and one asks whether a resonance identity of the
form \(\det(I-\mathcal B)=0\iff\) loss of stationary closure invertibility
survives beyond the scalar rank-one setting.

\section*{Declaration of Interest Statement}
The authors declare that they have no known competing financial interests or personal relationships that could have appeared to influence the work reported in this paper.

\section*{Funding sources}
This research did not receive any specific grant from funding agencies in the public, commercial, or not-for-profit sectors.

\section*{Declaration of generative AI use}
In writing this article, the author used Gemini~3 to polish the language. This content has been reviewed and edited by the author to ensure accuracy.

\section*{Data Availability Statement}
Data sharing is not applicable as this study does not analyze or generate new datasets.

\bibliographystyle{unsrt}
\bibliography{reference}

\appendix
\section{Frozen-flow stability estimate}
\label{app:frozen_flow_stability}
The transport equation carries a prescribed nonzero inflow flux at \(l_0\); the
attendant trace and interface handling is in the spirit of the analysis of
transport equations with inflow boundary conditions \cite{scott2022transport}, here
carried out by the method of characteristics at \(BV\) regularity.
Fix \(T>0\) and frozen paths \(e_i\in\mathcal K_T\) (\(i=1,2\)); in particular
\(e_i\in C^0([0,T];[0,E_T])\) with \(\operatorname{Lip}(e_i)\le\Lambda_T\), the
Lipschitz-in-time ball produced by the contraction argument of
\Cref{subsec:proof_wellposedness}. Set \(g_i(t,l):=g(e_i(t),l)\),
\(\mu_i(t,l):=\mu(e_i(t),l)\), and let \(z_i=z_{e_i}\) solve
\[
    \partial_t z_i+\partial_l(g_i z_i)=-(\mu_i+u)z_i,
    \qquad
    g_i(t,l_0)z_i(t,l_0)=p(t),
    \qquad
    z_i(0,l)=x_0(l),
\]
equivalently \(\partial_t z_i+g_i\partial_lz_i=-(\partial_lg_i+\mu_i+u)z_i\).
The goal is
\begin{equation}
    \label{eq:app_goal}
    \|z_{e_1}(t)-z_{e_2}(t)\|_{L^1(\Omega)}
    \le
    C_T\int_0^t |e_1(s)-e_2(s)|\,ds,
    \qquad
    0\le t\le T .
\end{equation}
Throughout, \(C_T>0\) may change from line to line and depends only on
\[
    T,\ g_m,\ C_{E_T},\ L_{E_T},\ \|x_0\|_{BV},\ M_u^{\rm dyn},\
    p_{\min},\ p_{\max},\ M_p(T),\ \Lambda_T,\ |\Omega|.
\]
\subsection{Flow and entrance-time stability}
Let \(L_i(\tau;t,l)\) be the backward characteristic
\(\frac{d}{d\tau}L_i=g_i(\tau,L_i)\), \(L_i(t;t,l)=l\); since \(g_i\ge g_m>0\),
characteristics move monotonically from \(l_0\) to \(l_m\). Define
\(\ell_i(t):=L_i(t;0,l_0)\) and the entrance time \(\tau_i(t,l)\) by
\(L_i(\tau_i;t,l)=l_0\), \(0<\tau_i\le t\), where applicable.
\begin{lemma}[Flow stability]
\label{lem:app_flow_stability}
For \(0\le\tau\le t\le T\),
\begin{equation}
    \label{eq:app_flow_stability}
    |L_1(\tau;t,l)-L_2(\tau;t,l)|
    \le
    C_T\int_\tau^t |e_1(s)-e_2(s)|\,ds,
\end{equation}
and consequently
\(|\ell_1(t)-\ell_2(t)|\le C_T\int_0^t|e_1-e_2|\).
\end{lemma}
\begin{proof}
Subtracting the flow equations and splitting
\(g(e_1,L_1)-g(e_2,L_2)=[g(e_1,L_1)-g(e_1,L_2)]+[g(e_1,L_2)-g(e_2,L_2)]\), the
first term is \(\le C_T|L_1-L_2|\) by the \(W^{2,\infty}\)-in-size bound and the
second \(\le C_T|e_1-e_2|\) by \eqref{eq:model_E_lip}. Thus
\(D(\tau):=|L_1-L_2|\) obeys \(D(t)=0\),
\(D(\tau)\le C_T\int_\tau^t D+C_T\int_\tau^t|e_1-e_2|\), and backward Gronwall
gives \eqref{eq:app_flow_stability}; the corner curve
\((\tau=0,l=l_0)\) gives the second bound.
\end{proof}
\begin{lemma}[Entrance-time stability]
\label{lem:app_entrance_stability}
If both backward characteristics enter through \(l_0\), then
\[
    |\tau_1(t,l)-\tau_2(t,l)|
    \le
    \frac{C_T}{g_m}\int_0^t |e_1(s)-e_2(s)|\,ds .
\]
\end{lemma}
\begin{proof}
Assume \(\tau_1\le\tau_2\), so \(L_1(\tau_1)=l_0=L_2(\tau_2)\) and \(L_1\)
remains in \(\Omega\) on \([\tau_1,t]\); evaluating the earlier-entering
characteristic at the later time, \(g_m(\tau_2-\tau_1)\le
L_1(\tau_2)-L_1(\tau_1)=L_1(\tau_2)-l_0=|L_1(\tau_2)-L_2(\tau_2)|\le
C_T\int_{\tau_2}^t|e_1-e_2|\), by \Cref{lem:app_flow_stability}. Divide by
\(g_m\); the case \(\tau_2\le\tau_1\) is identical.
\end{proof}

\subsection{A bounded-variation composition estimate}

The following composition estimate is a direct consequence of the
measure representation of the distributional derivative of a
one-dimensional \(BV\) function; see, e.g.,
Ambrosio--Fusco--Pallara
\cite{ambrosio2000functions}.
For completeness we include the short proof.
\begin{lemma}[\(BV\)-composition estimate]
\label{lem:app_bv_composition}
Let \(w\in BV(\Omega)\), and extend it to
\(\overline w\in BV(\mathbb R)\) by constant continuation,
so that
\[
\operatorname{TV}_{\mathbb R}(\overline w)
=
\operatorname{TV}_{\Omega}(w).
\]
Let
\(\phi_1,\phi_2:\Omega\to\Omega\)
be monotone bi-Lipschitz maps satisfying
$|\phi_i'|\ge J_m>0$ a.e., and
$\|\phi_1-\phi_2\|_{L^\infty(\Omega)}
\le d$.
Then
\[
\int_\Omega
|w(\phi_1(l))-w(\phi_2(l))|
\,dl
\le
\frac{d}{J_m}
\operatorname{TV}_\Omega(w).
\]
\end{lemma}

\begin{proof}
By the change of variables
\(\xi=\phi_1(l)\),
the lower bound on \(\phi_1'\) yields
$dl \le J_m^{-1}\,d\xi$,
hence
\[
\int_\Omega
|w(\phi_1(l))-w(\phi_2(l))|
\,dl
\le
\frac1{J_m}
\int_{\phi_1(\Omega)}
|\overline w(\xi)-\overline w(\psi(\xi))|
\,d\xi,
\]
where $\psi(\xi) := \phi_2(\phi_1^{-1}(\xi))$.
Since $|\psi(\xi)-\xi| \le d$,
it remains to estimate the last integral.
For the precise representative of a one-dimensional
\(BV\) function,
the distributional derivative
\(D\overline w\)
is a finite Radon measure and
\[
\overline w(b)-\overline w(a)
=
D\overline w((a,b])
\]
for every \(a<b\)
(see Ambrosio--Fusco--Pallara).
Therefore,
\[
|\overline w(\xi)-\overline w(\psi(\xi))|
\le
|D\overline w|
\!\left(
[\min\{\xi,\psi(\xi)\},
\max\{\xi,\psi(\xi)\}]
\right).
\]
Integrating over \(\phi_1(\Omega)\) and applying Fubini,
\[
\int_{\phi_1(\Omega)}
|\overline w(\xi)-\overline w(\psi(\xi))|
\,d\xi \le
\int_{\mathbb R}
\left(
\int_{\phi_1(\Omega)}
\mathbf 1_{
[\min(\xi,\psi(\xi)),
\max(\xi,\psi(\xi))]
}(y)
\,d\xi
\right)
\,d|D\overline w|(y).
\]
Since every interval
joining
\(\xi\)
and
\(\psi(\xi)\)
has length at most \(d\),
\[
\int_{\phi_1(\Omega)}
\mathbf 1_{
[\min(\xi,\psi(\xi)),
\max(\xi,\psi(\xi))]
}(y)
\,d\xi
\le
d,
\]
for every \(y\in\mathbb R\).
Hence
\[
\int_{\phi_1(\Omega)}
|\overline w(\xi)-\overline w(\psi(\xi))|
\,d\xi
\le
d\,|D\overline w|(\mathbb R)
=
d\,\operatorname{TV}_{\mathbb R}(\overline w).
\]
Since $\operatorname{TV}_{\mathbb R}(\overline w)
=
\operatorname{TV}_\Omega(w)$, the desired estimate follows.
\end{proof}

\subsection{The \(L^1\)-stability estimate}
For each \(i\), with \(q_i(t,l):=\partial_lg_i(t,l)+\mu_i(t,l)+u(t,l)\),
\begin{equation}
    \label{eq:app_initial_branch}
    z_i(t,l)
    =
    x_0(L_i(0;t,l))
    \exp\!\Big[-\!\int_0^t q_i(s,L_i(s;t,l))\,ds\Big],
    \qquad l>\ell_i(t),
\end{equation}
\begin{equation}
    \label{eq:app_boundary_branch}
    z_i(t,l)
    =
    \frac{p(\tau_i(t,l))}{g_i(\tau_i(t,l),l_0)}
    \exp\!\Big[-\!\int_{\tau_i(t,l)}^t q_i(s,L_i(s;t,l))\,ds\Big],
    \qquad l<\ell_i(t),
\end{equation}
with \(\|q_i\|_{L^\infty}\le C_T\), \(e^{-C_TT}\le\exp[-\int q_i]\le e^{C_TT}\).
\begin{proposition}[Frozen-path stability]
\label{prop:app_frozen_path_stability}
Under the assumptions of Section~\ref{sec:model}, for every \(T>0\),
\begin{equation}
    \label{eq:app_frozen_path_stability}
    \|z_{e_1}(t)-z_{e_2}(t)\|_{L^1(\Omega)}
    \le
    C_T\int_0^t|e_1(s)-e_2(s)|\,ds,
    \qquad
    0\le t\le T.
\end{equation}
\end{proposition}
\begin{proof}
Set \(\varepsilon_e(t):=\int_0^t|e_1-e_2|\). By \Cref{lem:app_flow_stability},
the mismatch strip
\(\mathcal M(t)=(\min\ell_i(t),\max\ell_i(t))\) has
\(|\mathcal M(t)|\le C_T\varepsilon_e(t)\); on it
\(|z_i|\le e^{C_TT}\max\{\|x_0\|_{L^\infty},p_{\max}/g_m\}\) (using
\(BV\hookrightarrow L^\infty\) in one dimension for \(x_0\)), so
\begin{equation}
    \label{eq:app_mismatch_estimate}
    \int_{\mathcal M(t)}|z_1-z_2|\,dl\le C_T\varepsilon_e(t).
\end{equation}
Common initial-line region \(\mathcal I(t)=(\max\ell_i(t),l_m)\). Write
\(z_i=X_iG_i\), \(X_i(l):=x_0(L_i(0;t,l))\),
\(G_i(l):=\exp[-\int_0^t q_i(s,L_i)]\), so
\(|z_1-z_2|\le|X_1-X_2|\,\|G_1\|_\infty+\|X_2\|_\infty|G_1-G_2|\). Here
\(\|G_i\|_\infty\le e^{C_TT}\) and
\(\|L_1(0;t,\cdot)-L_2(0;t,\cdot)\|_\infty\le C_T\varepsilon_e(t)\), so
\Cref{lem:app_bv_composition} gives
\(\int_{\mathcal I}|X_1-X_2|\le C_T\operatorname{TV}(x_0)\varepsilon_e(t)\le
C_T\varepsilon_e(t)\). For the exponentials, \(|e^{-A_1}-e^{-A_2}|\le
e^{C_TT}|A_1-A_2|\), \(A_i(l):=\int_0^t q_i(s,L_i)\), and
\[
    |A_1-A_2|
    \le
    \int_0^t|(\partial_lg_1+\mu_1)(s,L_1)-(\partial_lg_2+\mu_2)(s,L_2)|
    +
    \int_0^t|u(s,L_1)-u(s,L_2)|.
\]
The coefficient part is \(\le C_T(|L_1-L_2|+|e_1-e_2|)\) pointwise (using
\(\partial_{ll}g,\partial_l\mu\in L^\infty\) and the \(E\)-Lipschitz continuity
of \(\partial_lg,\mu\)), so by \Cref{lem:app_flow_stability} its
\(\mathcal I\)-integral is \(\le C_T\varepsilon_e(t)\); the control part, by
Fubini and \Cref{lem:app_bv_composition} with \(w=u(s,\cdot)\),
\(\le\int_0^t C_T\operatorname{TV}(u(s,\cdot))\|L_1-L_2\|_\infty\,ds\le
C_T M_u^{\rm dyn}\varepsilon_e(t)\). Hence
\begin{equation}
    \label{eq:app_initial_region_estimate}
    \int_{\mathcal I(t)}|z_1-z_2|\,dl\le C_T\varepsilon_e(t).
\end{equation}
Common boundary region \(\mathcal B(t)=(l_0,\min\ell_i(t))\). Write
\(z_i=P_iG_i^b\), \(P_i(l):=p(\tau_i(t,l))/g_i(\tau_i(t,l),l_0)\),
\(G_i^b(l):=\exp[-\int_{\tau_i}^t q_i(s,L_i)]\). Then
\(|P_1-P_2|\le g_m^{-1}|p(\tau_1)-p(\tau_2)|+p_{\max}|g_1(\tau_1,l_0)^{-1}-g_2(\tau_2,l_0)^{-1}|\).
The entrance-time map \(l\mapsto\tau_i(t,l)\) is monotone; differentiating
\(L_i(\tau_i(t,l);t,l)=l_0\) in \(l\) gives
\[
    \partial_\tau L_i(\tau_i;t,l)\,\partial_l\tau_i
    +
    \partial_l L_i(\tau_i;t,l)=0,
\]
and since \(\partial_\tau L_i(\tau_i;t,l)=g_i(\tau_i,L_i(\tau_i;t,l))=g_i(\tau_i,l_0)\),
\[
    |\partial_l\tau_i|
    =
    \frac{|\partial_lL_i(\tau_i;t,l)|}{g_i(\tau_i,l_0)}.
\]
Since \(g_m\le g_i\le C_T\) and the first variation satisfies
\[
    \frac{d}{d\tau}\partial_lL_i(\tau;t,l)
    =
    \partial_lg_i(\tau,L_i(\tau;t,l))\partial_lL_i(\tau;t,l),
    \qquad
    \partial_lL_i(t;t,l)=1,
\]
we see $e^{-C_TT}
    \le
    |\partial_lL_i(\tau;t,l)|
    \le
    e^{C_TT}$, and there is \(J_T>0\) with \(|\partial_l\tau_i|\ge J_T\) a.e. on \(\mathcal B(t)\).
Thus \Cref{lem:app_bv_composition} with \(w=p\in BV(0,T)\) and
\Cref{lem:app_entrance_stability} give
\(\int_{\mathcal B}|p(\tau_1)-p(\tau_2)|\le C_T\operatorname{TV}_{[0,T]}(p)
\|\tau_1-\tau_2\|_\infty\le C_T\varepsilon_e(t)\). For the \(g\)-factor,
\[
    |g_1(\tau_1,l_0)-g_2(\tau_2,l_0)|
    \le
    |g(e_1(\tau_1),l_0)-g(e_1(\tau_2),l_0)|
    +
    |g(e_1(\tau_2),l_0)-g(e_2(\tau_2),l_0)|;
\]
the second term is \(\le C_T|e_1(\tau_2)-e_2(\tau_2)|\), and, since the frozen
paths lie in the Lipschitz-in-time ball \(\mathcal K_T\),
\(|e_1(\tau_1)-e_1(\tau_2)|\le\Lambda_T|\tau_1-\tau_2|\), so the first term is
\(\le C_T|\tau_1-\tau_2|\), controlled by \Cref{lem:app_entrance_stability}.
Hence \(\int_{\mathcal B}|P_1-P_2|\le C_T\varepsilon_e(t)\). For \(G_i^b\), with
\(A_i^b(l):=\int_{\tau_i}^t q_i(s,L_i)\), \(|G_1^b-G_2^b|\le e^{C_TT}|A_1^b-A_2^b|\)
and
\(|A_1^b-A_2^b|\le\int_{\max\tau_i}^t|q_1(s,L_1)-q_2(s,L_2)|+C_T|\tau_1-\tau_2|\),
the first term treated as on \(\mathcal I\), the second by
\Cref{lem:app_entrance_stability}, so
\(\int_{\mathcal B}|G_1^b-G_2^b|\le C_T\varepsilon_e(t)\); with
\(|P_2|\le p_{\max}/g_m\),
\begin{equation}
    \label{eq:app_boundary_region_estimate}
    \int_{\mathcal B(t)}|z_1-z_2|\,dl\le C_T\varepsilon_e(t).
\end{equation}
Since, up to endpoints of measure zero,
\[
    \Omega=\mathcal I(t)\cup\mathcal M(t)\cup\mathcal B(t),
\]
summing the three estimates gives \eqref{eq:app_frozen_path_stability}.
\end{proof}
\begin{remark}[Corner compatibility]
\label{rem:app_corner_compatibility}
The two representations meet along \(l=\ell_i(t)=L_i(t;0,l_0)\), the image of
the corner \((0,l_0)\); their one-sided traces agree iff
\(x_0(l_0)=p(0)/g(e_i(0),l_0)\). If this fails the solution jumps across the
single curve \(\ell_i(t)\), a measure-zero set for each fixed \(t\), so the
\(L^1\)-solution is unaffected; the two-path mismatch strip has length
\(\le C_T\int_0^t|e_1-e_2|\) and its contribution is already included in
\eqref{eq:app_mismatch_estimate}.
\end{remark}

Estimate \eqref{eq:app_goal} supplies the contraction bound
\(|\mathcal T e_1(t)-\mathcal T e_2(t)|\le\|\chi\|_{L^\infty}
\|z_{e_1}(t)-z_{e_2}(t)\|_{L^1}\le C_T\int_0^t|e_1-e_2|\), hence
\(\|\mathcal T e_1-\mathcal T e_2\|_{C^0([0,T])}\le C_TT\|e_1-e_2\|_{C^0([0,T])}\),
and the closed-loop environmental path is obtained by short-time contraction on
\(\mathcal K_T\) followed by continuation using the mass estimate.

\newpage

\begin{figure}[htbp]
    \centering
    \includegraphics[width=0.9\textwidth]{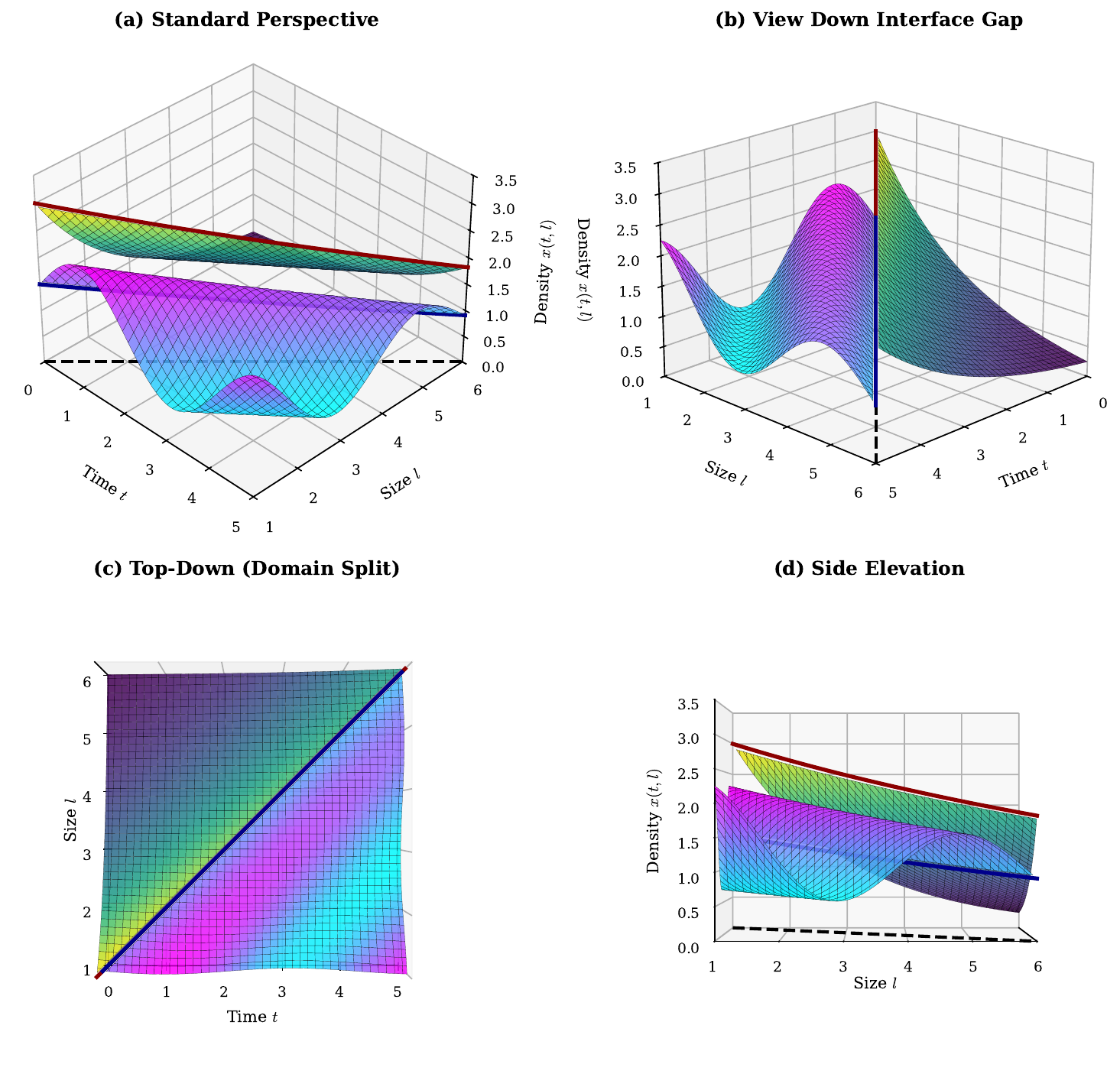}
    \caption{\label{fig:boundary_flux}}
\end{figure}

\begin{figure}[htbp]
    \centering
    \includegraphics[width=\textwidth]{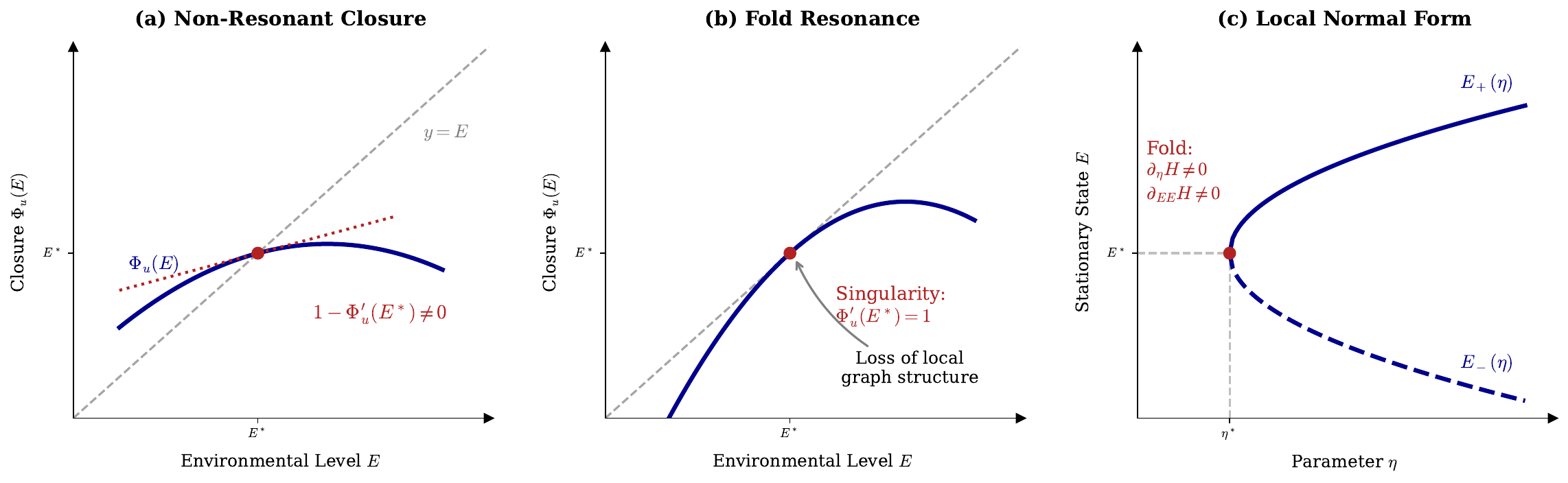}
    \caption{\label{fig:scalar_closure}}
\end{figure}

\begin{figure}[htbp]
    \centering
    \includegraphics[width=\textwidth]{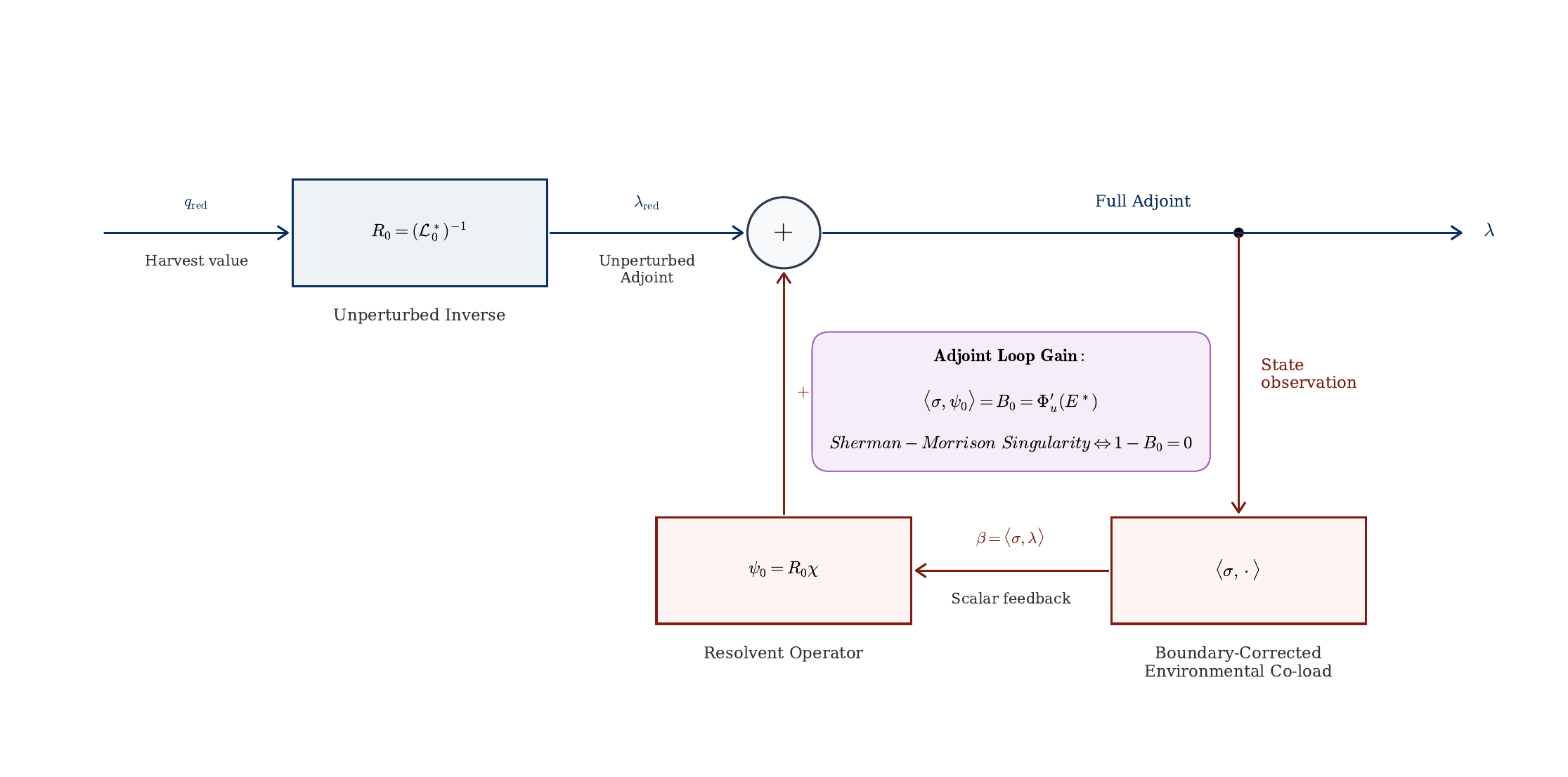}
    \caption{\label{fig:adjoint_loop}}
\end{figure}

\begin{figure}[htbp]
    \centering
    \includegraphics[width=0.85\textwidth]{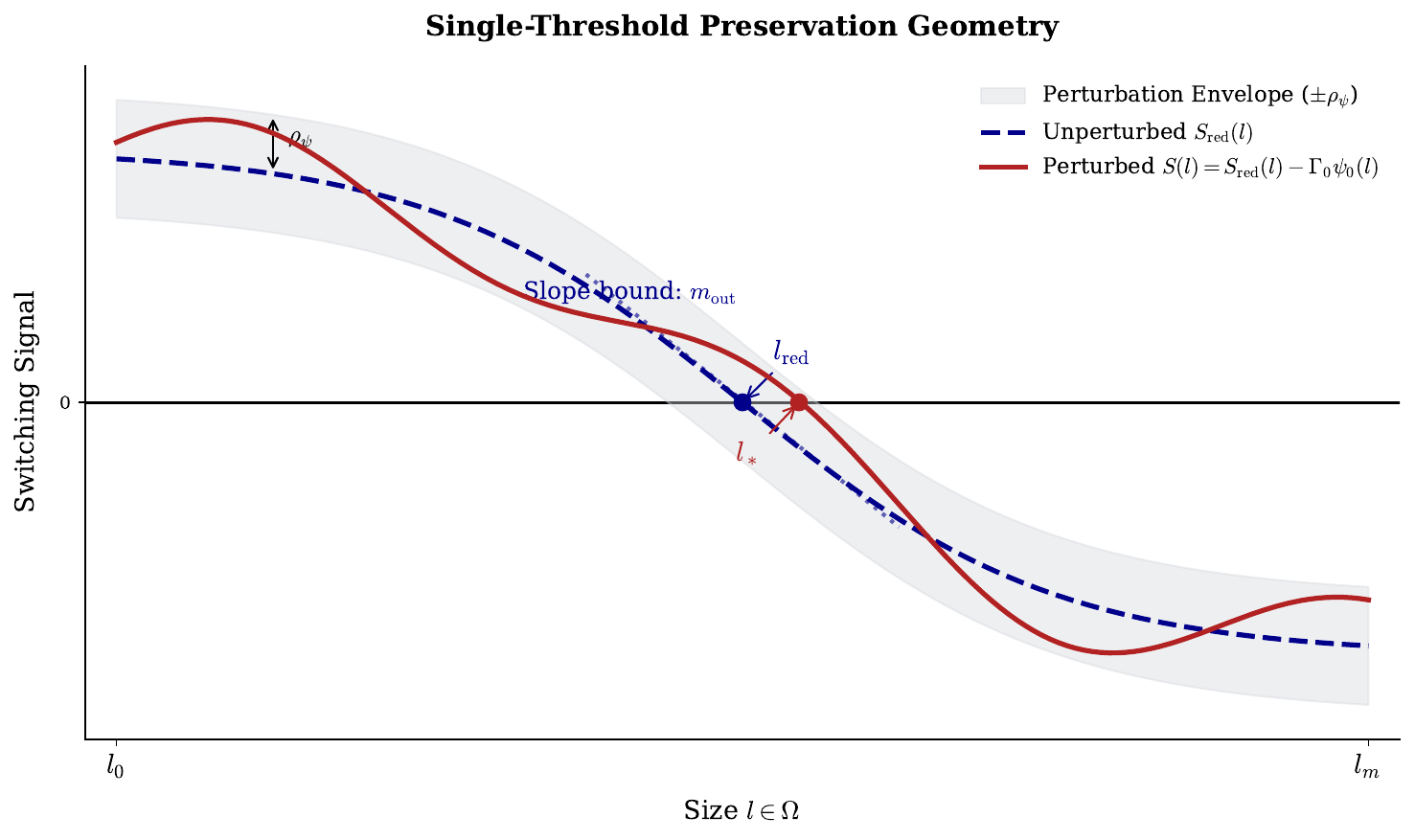}
    \caption{\label{fig:single_threshold}}
\end{figure}

\begin{figure}[htbp]
    \centering
    \includegraphics[width=0.85\textwidth]{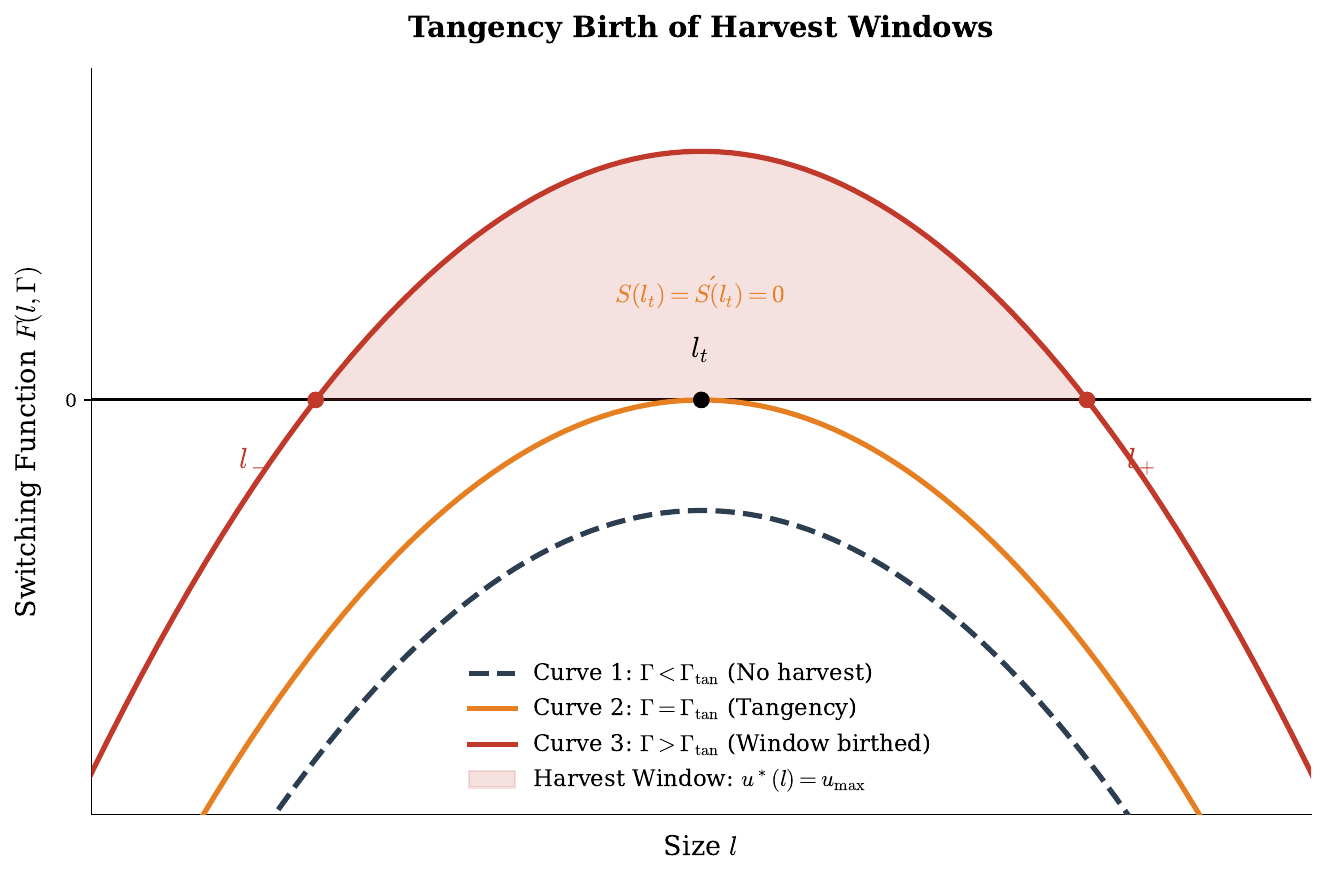}
    \caption{\label{fig:tangency_birth}}
\end{figure}
\clearpage

\newpage

\section*{Figure Captions}

\begin{enumerate}[label=(\roman*)]
    \item[Figure~\ref{fig:boundary_flux}] The non-standard boundary flux and characteristic flow in the $(t, l)$ plane. The subfigures demonstrate the delicate interface $\ell(t)$ created by the prescribed flux boundary condition $g(E(t), l_0)x(t, l_0) = p(t)$, which explicitly separates the initial-data-driven region from the boundary-data-driven region.
    \item[Figure~\ref{fig:scalar_closure}] Scalar closure and fold resonance. (a) A transversal, non-resonant intersection where $1 - \Phi_u'(E^*) \neq 0$. (b) A tangent intersection illustrating the fold resonance ($\Phi_u'(E^*) = 1$), which leads to the loss of local graph structure. (c) A parabolic fold branch extending into the $(\eta, E)$ plane, visually capturing the stationary states $E_\pm(\eta)$ and the resulting bifurcation.
    \item[Figure~\ref{fig:adjoint_loop}] The rank-one adjoint feedback loop framework. The primary feedforward path illustrates the unperturbed adjoint inversion $R_0 = (\mathcal L_0^*)^{-1}$ mapping the harvest value $q_{\rm red}$ to the unperturbed adjoint $\lambda_{\rm red}$. The closed-loop feedback path incorporates the boundary-corrected environmental co-load $\langle\sigma, \cdot\rangle$ and the resolvent operator $\psi_0$ returning through a summing junction, with the explicit loop gain labeled as $B_0 = \Phi_u'(E^*)$.
    \item[Figure~\ref{fig:single_threshold}] Single-threshold preservation and switching geometry over the size domain $\Omega = [l_0, l_m]$. The unperturbed switching function $S_{\rm red}(l)$ crosses the zero axis transversally at $l_{\rm red}$ (subject to the bound $m_{\rm out}$). The environmental feedback perturbation $\Gamma_0\psi_0(l)$ is constrained within an explicit envelope bound $\pm\rho_\psi$. Consequently, the perturbed switching function $S(l) = S_{\rm red}(l) - \Gamma_0\psi_0(l)$ uniquely crosses zero at the shifted root $l_*$, preserving the bang-bang structure.
    \item[Figure~\ref{fig:tangency_birth}] Tangency birth of harvest windows. The switching function $F(l, \Gamma)$ is plotted against the size $l$ for three distinct perturbation levels. For $\Gamma < \Gamma_{\rm tan}$ (Curve 1), the function remains strictly negative, yielding no harvesting. At the critical threshold $\Gamma = \Gamma_{\rm tan}$ (Curve 2), a double root emerges at $l_t$ satisfying the tangency condition $S(l_t) = S'(l_t) = 0$. For $\Gamma > \Gamma_{\rm tan}$ (Curve 3), the switching curve transversally crosses the zero axis twice, birthing a new harvest window $[l_-, l_+]$ where the optimal control is fully active ($u^*(l) = u_{\max}$).
\end{enumerate}
\end{document}